\documentclass[a4paper,11pt]{article}
\usepackage{latexsym,amssymb,amsmath,amsthm,amsxtra,xypic}
\usepackage{stmaryrd}
\usepackage{amsfonts}
\usepackage{bbm}
\usepackage{amscd}
\usepackage{mathrsfs}
\usepackage[dvips]{graphicx}
\usepackage[all]{xy}

\newcommand{\allowpagebreak}
\allowdisplaybreaks

\newtheorem{thm}{Theorem}[section]

\newtheorem{pro}[thm]{Proposition}

\theoremstyle{definition}
\newtheorem{defi}[thm]{Definition}

\newcommand {\emptycomment}[1]{}

\newcommand{\pf}{\noindent{\bf Proof.}\ }

\newcommand{\lam}{\lambda}
\newcommand{\si}{\sigma}

\newcommand{\g}{{A}}
\newcommand{\frkg}{{A}}
\newcommand{\fg}{{A}}
\newcommand{\hg}{{\widehat{A}}}

\newcommand{\cg}{\cdot_{{A}}}
\newcommand{\chg}{\cdot_{\widehat{A}}}

\newcommand{\h}{{V}}
\newcommand{\fh}{{V}}
\newcommand{\frkh}{{V}}

\newcommand{\gh}{{{A}\oplus {V}}}

\newcommand{\hh}{{M}}
\newcommand{\pp}{{A}}

\newcommand{\ulam}{^\lambda}
\newcommand{\dlam}{_\lambda}

\newcommand{\dM}{\mathrm{d}}
\newcommand{\E}{\mathrm{E}}

\newcommand{\Hom}{\mathrm{Hom}}

\newcommand{\Ker}{\mathrm{Ker}}

\newcommand{\End}{\mathrm{End}}

\newcommand{\ad}{\mathrm{ad}}

\newcommand{\id}{\mathrm{id}}

\newcommand{\trl}{\triangleleft}
\newcommand{\trr}{\triangleright}

\allowdisplaybreaks

\begin{document}

\title{ {Deformations and extensions of homotopy associative algebras
\author{\vspace{2mm} Tao Zhang}
 }}

\date{}
\maketitle



\begin{abstract}
The representation  and the cohomology theory of associative 2-algebras are developed.
We study the deformations and abelian extensions of associative 2-algebras in details.
\end{abstract}

\section{Introduction}
Strong homotopy associative algebras which are also called $A_\infty$-algebras were first introduced by Stasheff \cite{S1,S2}.
See \cite{K1,Lu} for  $A_\infty$-algebras in representation theory.
On the other direction, strong homotopy Lie algebras were studied in \cite{LM,LS}.
For the 2-term cases, Lie 2-algebras was investigated in details in \cite{Baez,LSZ}.

In this paper, we introduce the representation and the second cohomology theory for associative 2-algebras. We then study the deformations and abelian extensions of associative 2-algebras using this theory. It is proven that the second cohomology group classifies any abelian extension of associative 2-algebras. Additionally, we introduce the concept of crossed modules over associative algebras, which is equivalent to a strict associative 2-algebra. We believe that this new concept is of interest in its own right. Following the same approach as with associative 2-algebras, we develop the representation and the second cohomology theory of crossed modules over associative algebras. Finally, we thoroughly investigate the deformations and abelian extensions of crossed modules over associative algebras.

The paper is organized as follows. In Section 2, we recall the notions of associative 2-algebras and representation. We then study the cohomology of associative 2-algebras in low dimension, and give precise formulas for the coboundary operator. In Section 3, we study the 1-parameter infinitesimal deformation of an associative 2-algebra $\g$, and give the conditions on generating a 1-parameter infinitesimal deformation as a cocycle condition.  In Section 4, we study abelian extensions of associative 2-algebras. We prove that any abelian extension corresponds to a representation and a 2-cocycle.  We also classify abelian extensions by the second cohomology group. In Section 5, we pay attention to the special case of strict associative 2-algebras. We prove that this is equivalent to the concept of crossed modules over associative algebras.  At last, the  representation and the cohomology theory of  crossed modules over associative algebras are given.

\section{Homotopy associative algebras}\label{sec:cohom}

Recall that an associative algebra is a vector space ${A}$  endowed with a  map ${A}\otimes{A}\longrightarrow{A}$
satisfying
$$
(x\cdot y)\cdot  z =x\cdot (y\cdot  z)
$$
for all $x,y,z\in {A}.$

For an associative algebra ${{A}}$, a representation of  ${{A}}$ is a vector space ${M}$ together with two  maps
$$\cdot: {{A}}\otimes {M}\to {M} \,\,\, \text{and}\,\,\, \cdot: {M}\otimes {{A}}\to {M},$$
satisfying the following three axioms
\begin{itemize}
\item[$\bullet$]  $(x\cdot y)\cdot m=x\cdot (y\cdot  m)$,
\item[$\bullet$]  $x\cdot (m\cdot y)=(x\cdot  m)\cdot  y$,
\item[$\bullet$]   $m\cdot (x\cdot y)=(m\cdot x)\cdot y.$
\end{itemize}

The Hochschild cohomology of ${A}$ with coefficients in ${M}$ is the cohomology of the cochain complex
$C^n({A},{M})=\Hom(\otimes^n{A},{M})$ with the coboundary operator $d:C^n({A},{M})\longrightarrow
C^{n+1}({A},{M})$ defined by
\begin{eqnarray}
\nonumber&&d f^n(x_1,\dots,x_{n+1})\\
\nonumber&=&x_1\cdot f^n(x_2,\dots,x_{n+1})+(-1)^{n+1}f^n(x_1,\dots,x_k)\cdot x_{n+1}\\
\label{formulapartial}&&+\sum_{i=1}^{n}(-1)^i f^n(x_1,\dots,,x_{i-1},x_i\cdot x_{i+1},\dots,x_{n+1}).
\end{eqnarray}
It is proved that $d\circ d=0$ .

\begin{defi}\label{defi:Leibniz2}
  An associative 2-algebra $\g=\g_{{1}}\oplus \g_{0}$ consists of the following data:
\begin{itemize}
\item[$\bullet$] a complex of vector spaces $l_1^\g=\dM_\g:\g_{{1}}{\longrightarrow}\g_0,$ 

\item[$\bullet$] bilinear maps $l_2^\g=\cg:\g_i\otimes \g_j\longrightarrow
\g_{i+j}$, where $0\leq i+j\leq 1$,

\item[$\bullet$] a  trilinear map $l_3^\g:\g_0\otimes \g_0\otimes \g_0\longrightarrow
\g_{{1}}$,
   \end{itemize}
   such that for any $x,y,z,t\in \g_0$ and $a,b\in \g_{{1}}$, the following equalities are satisfied:
\begin{itemize}
\item[$\rm(a)$] $\dM_\g (x\cg a)=x\cg \dM_\g a,$
\item[$\rm(b)$]$\dM_\g (a\cg x)=\dM_\g a\cg x,$
\item[$\rm(c)$]$\dM_\g a\cg b=a\cg\dM_\g b,$
\item[$\rm(d)$]$\dM_\g l_3^\g(x,y,z)=(x\cg y)\cg z-x\cg (y\cg z),$
\item[$\rm(e_1)$]$ l_3^\g(x,y,\dM_\g a)=(x\cg y)\cg a-x\cg (y\cg a),$
\item[$\rm(e_2)$]$ l_3^\g(x,\dM_\g a,y)=(x\cg a)\cg y-x\cg (a\cg y),$
\item[$\rm(e_3)$]$ l_3^\g(\dM_\g a,x,y)=(a\cg x)\cg y-a\cg (x\cg y),$
\item[$\rm(f)$] the Stasheff identity:
\begin{eqnarray*}
x\cg l_3^\g(y,z,t)+l_3^\g(x,y,z)\cg t&=&l_3^\g(x\cg y,z,t)-l_3^\g(x,y\cg z,t)+l_3^\g(x,y,z\cg t).
\end{eqnarray*}
\end{itemize}
\end{defi}
We usually denote an associative 2-algebra by $(\frkg;\dM_\g,\cg,l_3^\g)$, or simply by $(\frkg;\dM,\cdot,l_3)$, $A$.

When $\dM=0$ and $l_3=0$, from Condition $(d)$ Definition \ref{defi:Leibniz2}, we get that $\fg_0$ is an associative algebra, from $(e_1)$, $(e_2)$ and $(e_3)$ in Definition \ref{defi:Leibniz2}, we obtain that  $\fg_1$ is a representation of  associative algebra $\fg_0$.
In general, $\dM\neq 0$ and $l_3\neq  0$, thus an associative 2-algebra is also called a homotopy  associative algebra.

\begin{defi}\label{defi:Lie-2hom}
Let $(\frkg;\dM_\g,\cg,l_3^\g)$ and $(\frkg';\dM',\cdot_{\g'},l_3')$ be associative 2-algebras. A
associative 2-algebra homomorphism $F$ from $\frkg$ to $ \frkg'$ consists of  linear maps $F_0:\frkg_0\rightarrow \frkg_0',~F_1:\frkg_{{1}}\rightarrow \frkg_{{1}}'$
 and $F_{2}: \frkg_{0}\times \frkg_0\rightarrow \frkg_{{1}}'$,
such that the following equalities hold for all $ x,y,z\in \frkg_{0},
a\in \frkg_{{1}},$
\begin{itemize}
\item [$\rm(i)$] $F_0\dM_\g=\dM'F_1$,
\item[$\rm(ii)$] $F_{0}(x\cg y)-F_{0}(x)\cdot_{\g'}F_{0}(y)=\dM'F_{2}(x,y),$
\item[$\rm(iii)_1$] $F_{1} (x\cg a)-F_{0}(x)\cdot_{\g'} F_{1}(a)=F_{2}(x,\dM_\g (a))$,
\item[$\rm(iii)_2$] $F_{1}(a\cg x)-F_{1}(a)\cdot_{\g'} F_{0}(x)=F_{2}(\dM_\g (a),x)$,
\item[$\rm(iv)$] $F_1(l_3^\g (x,y,z))-l_3'(F_0(x),F_0(y),F_0(z))$\\
  $=F_2(x, y\cg z)- F_2(x\cg y,z) + F_0(x)\cdot_{\g'} F_2(y,z) - F_2(x,y)\cdot_{\g'} F_0(z).$
\end{itemize}
 If $F_2=0$, the homomorphism $F$
is called a strict homomorphism.
\end{defi}

The identity homomorphism $1_\frkg: \frkg\to \frkg$ has the identity chain map together with $(1_\frkg)_2=0$.

\begin{defi}\label{ass-der-inf-def}
Let $(A_1 \xrightarrow{d} A_0, \cdot, {l}_3)$ be a $2$-term $A_\infty$-algebra. A homotopy derivation of degree $0$ on it consists of a chain map ${D} : A \rightarrow A$ (which consists of linear maps ${D}_i : A_i \rightarrow A_i$, for $i = 0 ,1$ satisfying ${D}_0 \circ d = d \circ {D}_1$) and ${D}_2 : A_0 \times A_0 \rightarrow A_1$ satisfying the followings: for any $x,y,z \in A_0$ and $m \in A_1$,
\begin{itemize}
\item[(a)] $ {D}_0 (x)\cdot y +x\cdot {D}_0 (b) - {D}_0 (x\cdot y)=d \circ {D}_2 (x, y),$
\item[(b)] ${D}_0 (x)\cdot m + x\cdot {D}_1 (m)  - {D}_1 (x\cdot m)={D}_2 (x, dm),$
\item[(c)] $ {D}_1 (m)\cdot x + m\cdot {D}_0 (x) - {D}_1 (m\cdot x)={D}_2 (dm, x),$
\item[(d)] $ {l}_3 ({D}_0 x, y, z) + {l}_3 (x, {D}_0 y, z) + {l}_3 (x, y, {D}_0 z)-{D}_1 \circ {l}_3 (x, y, z)= {D}_2 (x\cdot y, z)- {D}_2 (x, y\cdot z) + {D}_2 (x, y)\cdot z - x\cdot {D}_2 (y, z).$
\end{itemize}
\end{defi}

Let $F:\frkg\to \frkg'$ and $G:\frkg'\to \frkg''$ be two homomorphisms, then their composition $GF:\frkg\to \frkg''$
is a homomorphism defined as $(GF)_0=G_0\circ F_0:\frkg_0\to \frkg''_0$, $(GF)_1=G_1\circ F_1:\frkg_{{1}}\to \frkg''_{{1}}$
and
$$(GF)_2=G_2\circ (F_0\times F_0)+G_1\circ F_2:\frkg_0\times \frkg_0\to \frkg''_{{1}}.$$

Let $ {V}:V_{{1}}\stackrel{\partial}{\longrightarrow} V_0$ be a 2-term complex of vector spaces.
Then we get a new 2-term complex of vector spaces $\End({V}):\End^1(\mathbb
V)\stackrel{\delta}{\longrightarrow} \End^0_\partial({V})$ by
defining $\delta(A)=\partial\circ A+A\circ\partial$ for all
$A\in\End^1({V})$, where $\End^1({V})=\Hom(V_0,V_{{1}})$
and
$$\End^0_\partial({V})=\{X=(X_0,X_1)\in \Hom(V_0,V_0)\oplus \Hom(V_{{1}},V_{{1}})|~X_0\circ \partial=\partial\circ X_1\}.$$
Define  $l_2:\wedge^2 \End({V})\longrightarrow \End(\mathbb
V)$ by setting:
\begin{equation}
\left\{\begin{array}{l}l_2(X,Y)=XY,\\
l_2(X,A)=XA,\\
l_2(A,A')=0,\end{array}\right.\nonumber
\end{equation}
 for all $X,Y\in
\End^0_\partial({V})$ and $A,A'\in \End^1({V})$.
With the above notations, $(\End({V}),\delta,l_2)$ is a strict associative $2$-algebra.

\begin{defi}
A representation or bimodule of an associative 2-algebra $(\frkg; \dM_\g,\cg,l_3^\g)$ on a 2-term complex ${{V}}: \h_{{1}}\overset{\dM_\h}{\longrightarrow}\h_{0}$ consists of an action of $\fg_0$ on $\fh$,  two bilinear maps
$$\trr:\frkg_1\otimes V_0\longrightarrow V_1, a\otimes u\mapsto a\trr u,
\quad \trl: V_0\otimes \frkg_1\longrightarrow V_1, u\otimes a\mapsto u\trl a$$
and three trilinear maps
$$\fg_0\otimes\fg_0\otimes \fh_0\longrightarrow \h_1, x\otimes y\otimes u\mapsto (x,y)\trr u, $$
$$\fg_0\otimes\fh_0\otimes \fg_0\longrightarrow \h_1, x\otimes u\otimes y\mapsto x\trr u\trl y,$$
$$\fh_0\otimes\fg_0\otimes \fg_0\longrightarrow \h_1, u\otimes x\otimes y \mapsto  u\trl (x,y)$$
satisfying the following axioms:
 \begin{eqnarray*}
  &&(x\cg y)\trr u-x\trr (y\trr u)=\dM_{\frkh}((x,y)\trr u),\\
  &&(x\cg y)\trr m-x\trr (y\trr m)=(x,y)\trr \dM_{\frkh}m,\\
  &&( x\cg a)\trr u-x\trr(a\trr  u)=(x,\dM_\g (a))\trr u,\\
  &&( a\cg x)\trr u-a\trr(x\trr  u)=(\dM_\g (a),x)\trr u,
 \end{eqnarray*}
  \begin{eqnarray*}
  &&(x\trr u)\trl y-x\trr (u\trl y)=\dM_{\frkh}(x\trr u\trl y),\\
&&(x\trr m)\trl y-x\trr (m\trl y)=x\trr (\dM_{\frkh}m) \trl y),\\
  &&(x\trr u)\trl a-x\trr (u\trl a)=x\trr u\trl \dM_{\frkh}(a),\\
  &&(a\trr u)\trl y-a\trr (u\trl y)= \dM_{\frkh}(a) \trr u\trl y,
\end{eqnarray*}
 \begin{eqnarray*}
  &&u\trl (x\cg y)-(u\trl x)\trl y=\dM_{\frkh}(u\trl (x,y)),\\
  &&m\trl (x\cg y)-(m\trl x)\trl y=\dM_{\frkh}m\trl (x,y),\\
  &&u\trl (x\cg a)-(u\trl x)\trl a=(u\trl (x,\dM_{\fg} a)),\\
    &&u\trl (a\cg x)-(u\trl a)\trl x=(u\trl (\dM_{\fg} a, x)),
 \end{eqnarray*}
\begin{eqnarray}
x\trr (u  \trl (y\cdot z)) +(x \trr  u\trl y) \trl z= (x\trr u)  \trl (y, z)-x\trr (u\trl y)\trl z+x\trr u\trl (y\cdot z),
 \end{eqnarray}
\begin{eqnarray}
 x\trr (y\trr u\trl z) +((x,y)\trr u)\trl z =(x\cdot y)\trr u\trl z  -x\trr (y\trr u)\trl z + (x,y)\trr (u\trl z),
  \end{eqnarray}
\begin{eqnarray}
x\trr((y,z)\trr u)+l_3^\g (x,y,z)\trr u =(x\cdot y,z)\trr u -(x, y\cdot z)\trr u+ (x,y)\trr (z\trr u),
\end{eqnarray}
\begin{eqnarray}
u\trl l_3(x,y,z) + (u\trl (x,y))\trl z= (u\trl x) \trl (y,z) -u\trl (x\cdot y,z)+u\trl (x, y\cdot z).
 \end{eqnarray}
\end{defi}


For any associative 2-algebra $(\frkg;\dM_\g,\cg,l_3^\g)$, there is a natural {\bf adjoint representation} on itself. The corresponding representation $\ad=(\ad^L_0,\ad^R_0,\ad^L_1,\ad^R_1,\ad^L_2,\ad^M_2,\ad^R_2)$ is given by
\begin{eqnarray*}
&&\ad^L_0(x)(y+b)=x\trr(y+b)=x\cg (y+b),\\
&& \ad^R_0(x)(y+b)=(y+b)\trl x=(y+b)\cg x,\\
&&\ad^L_1(a)x=a\trr x=a\cg x,\quad \ad^R_1(a)x=x\trl a= x\cg a,\\
&&\ad^L_2(x,y)z=(x,y)\trr z=l_3^\g (x,y,z),\quad \ad^M_2(x,y)z=x\trr z\trl y=l_3^\g (x,z,y),\\
&&\ad^R_2(x,y)z=z\trl (x,y)=l_3^\g (z,x,y).
\end{eqnarray*}

Next we develop a cohomology theory for an associative 2-algebra $\fg$ with coefficients in a representation $({{V}},\rho)$.
The $k$-cochian  $C^{k}(\fg,{{V}})$ is defined to be the space:
\begin{eqnarray}
\Hom(\otimes^{k} \fg_0,V_0)\oplus\Hom(\mathcal{A}^{k-1},V_1)\oplus \Hom(\mathcal{A}^{k-2},V_0)\oplus \Hom(\fg_1^{k-1},V_1)\oplus \cdots
\end{eqnarray}
where
\begin{eqnarray*}
 &&\Hom(\mathcal{A}^{k-1},V_1):=\bigoplus_{i=1}^k  \Hom(\underbrace{{\fg_0}\otimes\cdots\otimes {\fg_0}}_{i-1}\otimes \fg_1\otimes \underbrace{{\fg_0}\otimes\cdots\otimes {\fg_0}}_{k-i},V_1),
 \end{eqnarray*}
is the direct sum of all possible tensor powers of ${\fg_0}$ and $\fg_1$ in which ${\fg_0}$ appears $k-1$ times but $\fg_1$ appears only once  and similarly for $\Hom(\mathcal{A}^{k-2},V_0).$

The cochain complex is given by
\begin{equation} \label{eq:complex}
\begin{split}
 &  V_0\stackrel{\mathbf{D}_0}{\longrightarrow} \Hom(\fg_0,V_0)\oplus\Hom(\fg_1,V_1)\oplus\Hom(\otimes^2\frkg_{0},V_{{1}})\stackrel{\mathbf{D}_1}{\longrightarrow}\\
 & \Hom(\fg_1,V_0)\oplus \Hom({\fg_0}\otimes {\fg_0}, V_0)\oplus \Hom(\mathcal{A}^{1}V_1)\oplus \Hom(\otimes^3\frkg_{0},V_{{1}})\stackrel{\mathbf{D}_2}{\longrightarrow}\\
  &\Hom(\otimes^2\fg_1,V_0)\oplus \Hom(\mathcal{A}^{1},V_0)\oplus
   \Hom(\otimes^3 \fg_0, V_0)\oplus \Hom(\mathcal{A}^{2}, V_1)\oplus
      \Hom(\otimes^4 \fg_0, V_1)\\
  & \stackrel{\mathbf{D}_3}{\longrightarrow}\cdots,
\end{split}
\end{equation}

Thus we get a cochain complex $C^k(\fg,{{V}})$ whose cohomology group
$$\mathbf{H}^k(\fg,{{V}})=Z^k(\fg,{{V}})/B^k(\fg,{{V}})$$
is defined as the cohomology group of $\fg$ with coefficients in ${{V}}$.
We only investigate in detail the first and the second cohomology groups which is needed in the following of this paper as follows.


We give the precise formulas for the coboundary operator in low dimension as follows.


For any 1-cochain $\Phi=(\phi,\varphi,\chi)$, where $\phi\in\Hom(\frkg_{0},V_{0})$, $\varphi\in \Hom(\frkg_{{1}},V_{{1}})$,  $\chi\in \Hom(\otimes^2\frkg_{0},V_{{1}})$,
it is called a $1$-cocycle if the following equations hold:
\begin{eqnarray}
\label{eq:1-cocycle1}&&\dM_\frkh \varphi(a)-\phi \dM_\g (a)=0,\\
\label{eq:1-cocycle2}&&x\trr \phi(y)-\phi(x)\trl y-\phi(x\cg y)+\dM_\frkh\circ \chi(x,y)=0,\\
\label{eq:1-cocycle3}&&x\trr \varphi(a)-\phi(x)\trl a-\varphi(x\cg a)+\chi(x,\dM_\g (a))=0,\\
\label{eq:1-cocycle4}&&a\trr \varphi(x)-\phi(a)\trl x-\varphi(a\cg x)+\chi(\dM_\g (a),x)=0,\\
\nonumber &&x\trr \chi(y,z)-\chi(x,y)\trl z+\chi(x, y\cdot z)- \chi(x\cdot y,z) \\
\label{eq:1-cocycle5}&&-\varphi(l_3^\g (x,y,z)+(x,y)\trr \phi(z)+x\trr\phi(y)\trl z+\phi(x)\trl (y,z)=0.
\end{eqnarray}


For any 2-cochain $(\psi,\omega,\mu,\nu,\theta)$, where $\psi\in \Hom(\frkg_{{1}},V_{0})$, $\omega\in\Hom(\otimes^2\frkg_{0},V_{0})$,
$\mu\in\Hom(\frkg_{{1}}\otimes\frkg_{0},V_{{1}})$, $\nu\in\Hom(\frkg_{0}\otimes\frkg_{{1}},V_{{1}})$, $\theta\in \Hom(\otimes^3\frkg_{0},V_{{1}})$, it is called a 2-coboundary if $(\psi,\omega,\mu,\nu,\theta)=D_1(\phi,\varphi,\chi)$.
It is called a 2-cocycle if the following equalities hold:
\begin{eqnarray}
\label{eq:coc01}&&x\trr \psi(a)-\psi( x\cg a)+\omega(x,\dM_\g (a))-\dM_V\mu(x,a)=0,\\
\label{eq:coc02}&&\psi(a)\trl x-\psi(a\cg x)+\omega(\dM_\g (a),x)-\dM_V\nu(a,x)=0,\\
\label{eq:coc03}&&a\trr \psi(b)+\nu(a,\dM_\g (b))-\psi(a)\trl b-\mu(\dM_\g (a),b)=0,\\[1em]
\label{eq:coc04}&&\omega(x,y)\trl z-x\trr \omega(y,z)+\omega(x\cg y,z)-\omega(x,y\cg z)\\
\nonumber&=&\partial\circ\theta(x,y,z)+\psi(l_3^\g (x,y,z)),\\[1em]
\nonumber&&\omega(x,y)\trl a-x\trr \mu(y,a)+\mu(x\cg y,a)-\mu(x,y\cg a)\\
\label{eq:coc05}&=&\theta(x,y,\dM_\g (a))+(x,y)\trr \psi(a),\\[1em]
\nonumber&&\omega(x,a)\trl y-x\trr \nu(a,y)+\nu( x\cg a,y)-\mu(x,(a\cg y))\\
\label{eq:coc06}&=&\theta(x,\dM_\g (a),y)+x\trr \psi(a) \trl y,\\[1em]
\nonumber&&\omega(a,x)\trl y-a\trr \omega(x,y)+\nu(a\cg x,y)-\nu(a,(x\cg y))\\
\label{eq:coc07}&=&\theta(\dM_\g (a),x,y)+\psi(a)\trl (x,y),\\[1em]
\nonumber&&x\trr \theta(y,z,t)+\theta(x,y,z)\trl t\\
\nonumber&&-\theta(x\cg y,z,t)+\theta(x,y\cg z,t)-\theta(x,y,z\cg t)\\
\nonumber&=&\mu(x,l_3^\g (y,z,t))+\nu(l_3^\g (x,y,z),t)\\
\label{eq:coc08}&&-(x,y)\trr \omega(z,t)-\omega(x,y)\trl (z,t)+x\trr \omega(y,z)\trl t.
\end{eqnarray}
The spaces of 2-coboundaries and 2-cocycles are denoted by $B^2(\fg, {{V}})$ and $Z^2(\fg, {{V}})$ respectively.
By direct computations, it is easy to see that the space of 2-coboundaries is contained in space of 2-cocycles, thus we obtain a second cohomology group  of an associative 2-algebra $\fg$ with coefficients in ${{V}}$  defined to be the quotient space
$$\mathbf{H}^2(\fg,  {{V}}))\triangleq Z^2(\fg, {{V}})/B^2(\fg, {{V}}).$$

\section{Infinitesimal Deformations of associative 2-algebras}

Let $(\frkg; \dM_\g, \cg, l_3^\g)$ be an associative 2-algebra, and $\psi:\g_{{1}}\rightarrow \g_0,~\omega:\otimes^2\g_{0}\rightarrow \g_0,~\nu:\g_0\otimes\g_{{1}}\rightarrow \g_{{1}},~\theta:\otimes^3\g_0\rightarrow\g_{{1}}$ be linear maps. Consider a $\lambda$-parametrized family of linear operations:
\begin{eqnarray*}
 \dM^\lambda (a)&\triangleq&\dM_\g (a)+\lambda\psi (a),\\
 {x\cdot\dlam y}&\triangleq& x\cg y+ \lambda\omega (x, y),\\
 {x\cdot\dlam a}&\triangleq&  x\cg a+ \lambda\mu (x, a),\\
 {a\cdot\dlam x}&\triangleq& a\cg x+ \lambda\nu (a,x),\\
 l_3^\lambda (x, y, z)&\triangleq& l_3^\g (x, y, z)+ \lambda\theta(x, y, z).
 \end{eqnarray*}

If all $(\dM^\lambda, \cdot\dlam, l^\lambda_3)$ endow the graded vector space $\g=\g_0\oplus\g_{{1}}$ with an associative 2-algebra structure, we say that
$(\psi,\omega,\mu, \nu,\theta)$ generates a 1-parameter infinitesimal deformation of associative 2-algebra $\frkg$.

\begin{thm}\label{thm:deformation}
$(\psi,\omega,\mu, \nu,\theta)$ generates a $1$-parameter infinitesimal deformation of the Lie $2$-algebra $\frkg$ is equivalent to the following conditions:
\begin{itemize}
  \item[\rm(i)] $(\psi,\omega,\mu, \nu,\theta)$ is a $2$-cocycle of $\frkg$ with the coefficients in the adjoint representation;

  \item[\rm(ii)] $(\psi,\omega,\mu, \nu,\theta)$ itself defines an associative 2-algebra structure on $\g$.
\end{itemize}
\end{thm}

\pf
If $(\frkg, \dM^\lambda, [\cdot,\cdot]\dlam, l^\lambda_3)$ is an associative 2-algebra, by (a) in Definition \ref{defi:Leibniz2}, we have
\begin{eqnarray*}
&&\dM\ulam (x\cdot\dlam a)-x \cdot\dlam \dM\ulam (a)\\
&=&(\dM_\g+\lambda\psi )( x\cg a+ \lambda\mu (x, a))-x\cg(\dM_\g (a)+\lambda\psi(a))-\lambda\omega (x, \dM_\g (a)+\lambda\psi(a))\\
&=&\dM_\g x\cg a+\lambda\{\psi  x\cg a+\dM_\g\mu (x,a)\}+\lambda^2\psi \omega (x,a)\\
&&-x\cg\psi(a)-\lambda\{\omega (x, \dM _\g a)+ x\cg\psi(a)\}-\lambda^2\mu (x,\psi(a))\\
&=&0,
\end{eqnarray*}
which implies that
\begin{eqnarray}
\label{eq:2-cocycle01'}\psi ( x\cg a)+\dM_\g \mu(x,a)-\omega (x, \dM_\g (a))- x\cg\psi(a)&=&0,\\
\label{eq:2-cocycle01''}\psi \mu (x,a)-\omega (x,\psi(a))&=&0.
\end{eqnarray}
Similarly, by (b) we have
\begin{eqnarray}
\label{eq:2-cocycle02'}\psi (a\cg x)+\dM_\g \nu(a,x)-\omega (\dM_\g (a),x)-\psi(a)\cg x &=&0,\\
\label{eq:2-cocycle02''}\psi \nu (a,x)-\omega (\psi(a),x)&=&0.
\end{eqnarray}
Now by (c) we have
\begin{eqnarray*}
&&\dM\ulam (a)\cdot\dlam b-a\cdot\dlam \dM\ulam (b)\\
&=&(\dM_\g (a)+\lambda\psi(a))\cg b+\lambda\nu(\dM_\g (a)+\lambda\psi(a),b)\\
&&-a\cg  (\dM_\g (b)+\lambda\psi(b))-\lambda\nu(a, \dM_\g (b)+\lambda\psi(b))\\
&=&\dM_\g (a)\cg b+\lambda\{\psi(a)\cg b+\nu (\dM_\g (a),b)\}+\lambda^2\nu (\psi(a),b)\\
&&-a\cg\dM_\g (b)-\lambda\{\nu (a, \dM_\g (b))+a\cg\psi(b)\}-\lambda^2\nu (a,\psi(b))\\
&=&0,
\end{eqnarray*}
which implies that
\begin{eqnarray}
\label{eq:2-cocycle03'}\psi(a)\cg b-a\cg\psi(b)+\mu (\dM_\g (a),b)-\nu (a,\dM_\g (b))&=&0,\\
\label{eq:2-cocycle03''}\mu (\psi(a),b)-\nu (a,\psi(b))&=&0.
\end{eqnarray}

Next, by (d) in Definition \ref{defi:Leibniz2}, we have
\begin{eqnarray*}
&&x\cdot\dlam(y\cdot\dlam z)-(x\cdot\dlam y)\cdot\dlam z+\dM\ulam l\ulam_3(x,y,z)\\
&=&x\cdot\dlam (y\cg z+\lambda\omega(y,z))-(x\cg y+\lambda\omega (x, y))\cdot\dlam z\\
&&+(\dM_\g +\lambda\psi )(l_3^\g+\lambda\theta)(x,y,z)\\
&=&x\cg (y\cg z)-(x\cg y)\cg z+\dM_\g l_3^\g (x,y,z)\\
&&+\lambda\left\{\omega (x,y\cg z)-\omega (x\cg y,z)+x\cg\omega(y,z)-\omega (x, y)\cg z\right.\\
&&\left.+\psi (l_3^\g (x,y,z))+\dM_\g \theta(x,y,z)\right\}\\
&&+\lambda^2\{\omega (x,\omega (y,z))-\omega (\omega (x,y),z)+\psi \theta(x,y,z)\}\\
&=&0,
\end{eqnarray*}
which implies that
\begin{eqnarray}
\notag&&\omega (x,y\cg z)-\omega (x\cg y,z)+x\cg \omega(y,z)-\omega (x, y)\cg z\\
\label{eq:2-cocycle11'}  &&+\psi (l_3^\g (x,y,z))+\dM_\g \theta(x,y,z)=0,\\
\label{eq:2-cocycle11''}&&\omega (x,\omega (y,z))-\omega (\omega (x,y),z)+\psi \theta(x,y,z)=0.
\end{eqnarray}
Similarly, by $\rm(e_1)$ we have
\begin{eqnarray}
\nonumber&&\mu (x,y\cg a)-\mu (x\cg y,a)+x\cg \mu(y,a)-\omega(x,y)\cg a\\
\label{eq:2-cocycle21'}&&+l_3^\g (x,y,\psi(a))+\theta(x,y,\dM_\g (a))=0,\\
\label{eq:2-cocycle21''}&&\mu (x,\mu(y,a))-\nu (\omega(x,y),a)+\theta(x,y,\psi(a))=0.
\end{eqnarray}
By $\rm(e_2)$ we have
\begin{eqnarray}
\nonumber&&\mu (x,a\cg y)-\nu ( x\cg a,y)+x\cg \nu(a,y)-\mu(x,a)\cg y\\
\label{eq:2-cocycle22'}&&+l_3^\g (x,\psi(a), y)+\theta(x,\dM_\g (a),y)=0,\\
\label{eq:2-cocycle22''}&&\mu (x,\nu(a,y))-\nu (\omega(x,a),y)+\theta(x,\psi(a), y)=0.
\end{eqnarray}
By $\rm(e_3)$ we have
\begin{eqnarray}
\nonumber&&\nu (a,y\cg x)-\nu (a\cg y,x)+a\cg \omega(y,x)-\nu(a,y)\cg x\\
\label{eq:2-cocycle23'}&&+l_3^\g (\psi(a),x,y)+\theta(\dM_\g (a),x,y)=0,\\
\label{eq:2-cocycle23''}&&\nu (a,\omega(y,x))-\nu (\nu(a,y),x)+\theta(\psi(a),x,y)=0.
\end{eqnarray}

At last, by (f) in Definition \ref{defi:Leibniz2}, we have
\begin{eqnarray}
\nonumber&&x\cg \theta(y,z,t)+\theta(x,y,x)\cg t\\
\nonumber &&-\theta(x\cg y,z,t)+\theta(x,y\cg z,t)-\theta(x,y,z\cg t)\\
\nonumber && =\mu(x,l_3^\g (y,z,t))+\nu(l_3^\g (x,y,z),t)\\
\label{eq:2-cocycle3'} &&+l_3^\g (x,\omega(y,z),t)-l_3^\g (x,y,\omega(z,t))-l_3^\g (\omega(x,y),z,t)\\[1em]
\nonumber&&\mu(x,\theta(y,z,t))+\nu(\theta(x,y,z),t)\\
\label{eq:2-cocycle3''}  &&-\theta(\omega(x,y),z,t)+\theta(x,\omega(y,z),t)-\theta(x,y,\omega(z,t))=0.
\end{eqnarray}
By \eqref{eq:2-cocycle01'}, \eqref{eq:2-cocycle02'}, \eqref{eq:2-cocycle11'},
\eqref{eq:2-cocycle21'}, \eqref{eq:2-cocycle22'}, \eqref{eq:2-cocycle23'} and \eqref{eq:2-cocycle3'},
we deduce that $(\psi,\omega,\mu,\nu,\theta)$ is a 2-cocycle of $\g$ with the coefficients in the adjoint representation.

Furthermore, by \eqref{eq:2-cocycle01''}, \eqref{eq:2-cocycle02''}, \eqref{eq:2-cocycle11''},
\eqref{eq:2-cocycle21''}, \eqref{eq:2-cocycle22''}, \eqref{eq:2-cocycle23''} and \eqref{eq:2-cocycle3''},
 $(\g;\psi, \omega, \mu, \nu, \theta)$ is an associative 2-algebra.
\qed

It can be shown that equivalent deformations give rise to the same 2-cocycle classes in $H^2(\fg,\fg)$.

\subsection{Nijenhuis operators}

In this subsection, we introduce the notion of Nijenhuis operators,
which could gives trivial deformations.

Let $(\frkg; \dM_\g, \cg, l_3^\g)$ be an associative 2-algebra.
Consider a second order deformation $(\fg_\lambda;\dM\ulam,\cdot\dlam, l^\lambda_3)$
\begin{eqnarray*}
 \dM^\lambda(a)&\triangleq&\dM_\g (a)+\lambda\psi (a),\\
 {x\cdot\dlam y}&\triangleq& (x\cg y)+ \lambda\omega (x, y),\\
 {x\cdot\dlam a}&\triangleq&  x\cg a+ \lambda\mu (x, a),\\
 {a\cdot\dlam x}&\triangleq& a\cg x+ \lambda\nu (a,x),\\
 l_3^\lambda (x, y, z)&\triangleq& l_3^\g (x, y, z)+ \lambda\theta_1(x, y, z)+ \lambda^2\theta_2(x, y, z).
 \end{eqnarray*}

\begin{defi}
An infinitesimal deformation is said to be trivial if there exists  linear maps $N_0:\frkg_0\to \frkg_0,N_1: \frkg_{{1}}\to \frkg_{{1}}$,
and $N_2:\otimes^2 \frkg_0\to \frkg_{{1}}$,
such that $(T_0,T_1,T_2)$ is a homomorphism from $(\fg_\lambda;\dM\ulam,\cdot_\lambda, l^\lambda_3)$ to $(\fg;\dM_\g,\cg, l^\g_3) $, where $T_0 = \id + \lambda N_0$, $T_1 = \id + \lambda N_1$ and $T_2 = \lambda N_2$.
\end{defi}

Note that $(T_0,T_1,T_2)$ is a homomorphism means that
\begin{eqnarray}
\label{eq:deform0} \dM_\g \circ T_1(a)&=&T_0\circ \dM\ulam(a),\\
\label{eq:deform1} T_0(x\cdot\dlam y)&=&T_0 (x)\cg T_0 (y)+\dM_\g T_2(x,y),\\
\label{eq:deform2} T_1(x\cdot\dlam a)&=&T_0 (x)\cg T_1 (a)+T_2(x,\dM\ulam a),\\
\label{eq:deform22}T_1(a\cdot\dlam x)&=&T_1 (a)\cg T_0 (x)+T_2(\dM\ulam a, x),\\
\nonumber T_1(l_3^\lam(x,y,z))&=&l_3^\g (T_0(x),T_0(y),T_0(z))+T_2(x, y\cdot_\lam z)- T_2(x\cdot_\lam y,z) \\
\label{eq:deform3} && + T_0(x)\cg  T_2(y,z)  - T_2(x,y)\cg T_0(z).
\end{eqnarray}
Now we consider conditions that $N=(N_0,N_1,N_2)$ should  satisfy.

For \eqref{eq:deform0}, we have
$$\dM_\g (a)+\lambda \dM_\g N_1(a)=\dM_\g (a)+\lambda N_0(\dM_\g (a))+\lambda \psi(a)+\lambda^2N_0\psi(a).$$
Thus, we have
$$\psi(a)=\dM_\g N_1(a)-N_0\dM_\g (a),$$
$$N_0\psi(a)=0.$$
It follows that $N$ must satisfy the following condition:
\begin{align}\label{Nij0}
N_0(\dM_\g N_1(a)-N_0\dM_\g (a))=0.
\end{align}

For \eqref{eq:deform1}, the left hand side is equal to
$$(x\cg y)+\lambda N_0(x\cg y)+\lambda\omega(x,y)+\lambda^2 N_0\omega(x, y),$$
and the right hand side is equal to \begin{eqnarray*}
(x\cg y)+\lambda N_0(x)\cg y+\lambda x\cg N_0(y)+\lambda^2 N_0(x)\cg N_0(y)+\lambda\dM_\g N_2(x,y).
\end{eqnarray*}
Thus, \eqref{eq:deform1} is equivalent to
$$\omega(x, y) = N_0(x)\cg y + x\cg N_0(y)- N_0(x\cg y)+\dM_\g N_2(x,y),$$
$$N_0\omega(x, y) = N_0(x)\cg N_0(y). $$
It follows that $N$ must satisfy the following condition:
\begin{align}\label{Nij1}
N_0(x)\cg N_0(y) - N_0(N_0(x)\cg y - N_0(x\cg N_0(y))+ N_0^2(x\cg y)-N_0\dM_\g N_2(x,y)=0.
\end{align}

For \eqref{eq:deform2},  the left hand side is equal to
$$ x\cg a+\lambda \mu(x,a) +\lambda N_1( x\cg a)+\lambda^2 N_1\mu(x,a),$$
and the right hand side is equal to
$$ x\cg a+\lambda N_0(x)\cg a+\lambda x\cg N_1(a)+\lambda^2 N_0(x)\cg N_1(a)+\lambda  N_2(x,\dM_\g (a))+\lambda^2N_2(x,\psi(a)).$$
Thus, \eqref{eq:deform2} is equivalent to
$$\mu(x, a) = N_0(x)\cg a +  x\cg N_1(a) - N_1(x\cg a)+ N_2(x,\dM_\g (a)),$$
$$N_1\mu(x, a) = N_0(x)\cg N_1(a)+N_2(x,\psi(a)). $$
It follows that $N$ must satisfy the following condition:
\begin{align}\label{Nij2}
N_0(x)\cg N_1(a)+N_2(x,\psi(a)) - N_1(N_0(x)\cg a) - N_1 (x\cg N_1(a)) + N_1^2(x\cg a)-N_1N_2(x,\dM_\g (a))=0.
\end{align}

Similarly, from \eqref{eq:deform22} we get
\begin{align}\label{Nij22}
N_1(a)\cg N_0(x)+N_2(\psi(a), x) - N_1 (a\cg N_0(x)) - N_1 N_1(a)\cg x + N_1^2(a\cg x)-N_1N_2(\dM_\g (a),x)=0.
\end{align}

For \eqref{eq:deform3}, the left hand side is equal to
\begin{eqnarray*}
&&l_3^\g (x,y,z)+\lambda\theta_1(x,y,z)+\lambda^2\theta_2(x,y,z)\\
&&+ \lambda N_1l_3^\g (x,y,z)+\lambda^2 N_1\theta_1(x,y,z)+\lambda^3N_1\theta_2(x,y,z)\\
&&+\lambda N_2(x\cg y,z)+c.p.+\lambda^2 N_2(\omega(x,y),z)+c.p.
\end{eqnarray*}
and the right hand side is equal to
\begin{eqnarray*}
&&l_3^\g (x,y,z)+\lambda l_3^\g (N_0(x),y,z)+c.p.+\lambda^2 l_3^\g (N_0(x),N_0(y),z)+c.p.\\
&&+\lambda^3l_3^\g (N_0(x),N_0(y),N_0z)+\lambda[x, N_2(y,z)]_\g+c.p.+\lambda^2 N_0(x)\cg  N_2(y,z)+c.p.
\end{eqnarray*}
Then we obtain
 \begin{eqnarray}
\notag\theta_1(x,y,z)&=&l_3^\g (N_0(x),y,z)+l_3^\g (x,N_0(y),z)+l_3^\g (x,y,N_0z)-N_1(l_3^\g (x,y,z))\\
\notag&&-N_2(x\cg y,z)-c.p.+x\cg N_2(y,z)+c.p.\\
\notag\theta_2(x,y,z)&=&l_3^\g (N_0(x),N_0(y),z)+l_3^\g (N_0(x),y,N_0z)+l_3^\g (x,N_0(y),N_0z)-N_1(\theta_1(x,y,z))\\
\notag&&-N_2(\omega(x,y),z)-c.p.+N_0(x)\cg N_2(y,z)+c.p.\\
  N_1\theta_2(x,y,z)&=& l_3^\g (N_0(x),N_0(y),N_0z).\label{Nij3}
\end{eqnarray}

Thus,  $(T_0,T_1,T_2)$ is a homomorphism if and only if  $N=(N_0,N_1,N_2)$ satisfy conditions \eqref{Nij0}, \eqref{Nij1}, \eqref{Nij2}, \eqref{Nij22},  \eqref{Nij3}.

  Define the following maps
 \begin{eqnarray*}
   \dM_N(a)&=&\dM_\g\circ N_1(a)-N_0\circ \dM_\g (a),\\
~x\cdot_N y&=&N_0(x)\cg y + x\cg N_0(y)- N_0(x\cg y),\\
~x\cdot_N a&=&N_0(x)\cg a +  x\cg N_1(a) - N_1(x\cg a),\\
~a\cdot_N x&=& N_1(a)\cg x +  a\cg N_0(x) - N_1(a)\cg x,\\
~l_3^N(x,y,z)&=&l_3^\g (N_0(x),y,z)+l_3^\g (x,N_0(y),z)+l_3^\g (x,y,N_0z)-N_1(l_3^\g (x,y,z)),\\
~l_3^{N^2}(x,y,z)&=&l_3^\g (N_0(x),N_0(y),z)+l_3^\g (N_0(x),y,N_0z)+l_3^\g (x,N_0(y),N_0z)-N_1(l_3^N(x,y,z)).
 \end{eqnarray*}
 \begin{defi}\label{defi:Nij}
  A map $N=(N_0,N_1)$ is called a Nijenhuis operator if for all $x,y,z\in\g_0$ and $a\in\g_{{1}}$, the following conditions are satisfied:
   \begin{itemize}
     \item[(i)] $N_0\dM_N(a)=0;$
     \item[(ii)] $N_0(x\cdot_N y)=N_0(x)\cg N_0(y);$
     \item[(iii)] $N_1(x\cdot_N a)=N_0(x)\cg N_1(a);$
          \item[(iv)] $N_1(a\cdot_N x)=N_1(a)\cg N_0(x);$
     \item[(v)] $N_1l_3^{N^2}(x,y,z)=l_3^\g (N_0(x),N_0(y),N_0z).$
   \end{itemize}
 \end{defi}

\begin{thm}\label{thm:Nijenhuis}
Let $N=(N_0,N_1)$ be a Nijenhuis operator. Then a second order deformation can be obtained by putting
\begin{equation}\label{Nijenhuis}
\left\{\begin{array}{rll}
\psi(a)&=&\dM_\g N_1(a)-N_0\dM_\g (a),\\
\omega(x, y)&=&N_0(x)\cg y + x\cg N_0(y)- N_0(x\cg y),\\
\mu(x, a)&=&N_0(x)\cg a+ x\cg N_1(a) - N_1(x\cg a),\\
\nu(a,x)&=&N_1(a)\cg x + a\cg  N_0(x) - N_1(a)\cg x,\\
\theta_1(x,y,z)&=&l_3^\g (N_0(x),y,z)+l_3^\g (x,N_0(y),z)+l_3^\g (x,y,N_0z)-N_1(l_3^\g (x,y,z)),\\
\theta_2(x,y,z)&=&l_3^\g (N_0(x),N_0(y),z)+l_3^\g (N_0(x),y,N_0z)+l_3^\g (x,N_0(y),N_0z)-N_1(\theta_1(x,y,z)).
\end{array}\right.
\end{equation}
\end{thm}

\section{ Abelian extensions of associative 2-algebras}

In this section, we study abelian extensions of associative 2-algebras using the cohomology theory. We show that associated to any abelian extension, there is a representation and a 2-cocycle.
Furthermore, abelian extensions can be classified by the second cohomology group.

\begin{defi}
Let $(\fg;\dM_\g,\cg,l_3^\g)$,  $(\widehat{\fg };\widehat{\dM},\chg,\widehat{l_3})$ be associative 2-algebras and
$i=(i_{0},i_{1}):\h\longrightarrow\widehat{\fg },~~p=(p_{0},p_{1}):\widehat{\fg }\longrightarrow\fg $
be strict homomorphisms. The following sequence of associative 2-algebras is a
short exact sequence if $\mathrm{Im}(i)=\mathrm{Ker}(p)$,
$\mathrm{Ker}(i)=0$ and $\mathrm{Im}(p)=\g$.

\begin{equation}\label{eq:ext1}
\CD
  0 @>0>>  \h_{{1}} @>i_1>> \widehat{\g}_{{1}} @>p_1>> \g_{{1}} @>0>> 0 \\
  @V 0 VV @V \dM_\h VV @V \widehat{\dM} VV @V\dM_\g VV @V0VV  \\
  0 @>0>> \h_{0} @>i_0>> \widehat{\g}_0 @>p_0>> \g_0@>0>>0
\endCD
\end{equation}
We call $\widehat{\fg }$  an extension of $\fg $ by
$\h$, and denote it by $\E_{\widehat{\g}}.$
It is called an abelian extension if $\h$ is abelian, i.e. if $\cdot_{\h}=0$ and $l^{\h}_3(\cdot,\cdot,\cdot)=0$. 
We will view $\h$ as subcomplex of $\hg$ directly, and omit the map $i$.
\end{defi}

 A splitting $\sigma:\fg \longrightarrow\widehat{\fg }$ of $p:\widehat{\fg }\longrightarrow\fg $
consists of linear maps
$\sigma_0:\fg_0\longrightarrow\widehat{\g}_0$ and
$\sigma_1:\fg_{{1}}\longrightarrow\widehat{\g}_{{1}}$
 such that  $p_0\circ\sigma_0=id_{\fg_0}$ and  $p_1\circ\sigma_1=id_{\fg_{{1}}}$.

 Two extensions of associative 2-algebras
 $\E_{\widehat{\g}}:0\longrightarrow\frkh\stackrel{i}{\longrightarrow}\widehat{\g}\stackrel{p}{\longrightarrow}\g\longrightarrow0$
 and $\E_{\tilde{\g}}:0\longrightarrow\frkh\stackrel{j}{\longrightarrow}\tilde{\g}\stackrel{q}{\longrightarrow}\g\longrightarrow0$ are equivalent,
 if there exists an associative $2$-algebra homomorphism $F:\widehat{\g}\longrightarrow\tilde{\g}$  such that $F\circ i=j$, $q\circ
 F=p$ and $F_2(i(u),\alpha)=0$, for any
 $u\in\frkh_0,~\alpha\in\widehat{\g}_0$.

Let $\widehat{\fg }$ be an abelian extension of $\fg $ by
$\frkh$, and $\sigma:\g\longrightarrow\widehat{\g}$ be a splitting.
Define the representation by
\begin{equation}\label{abelian:rep}
\left\{\begin{array}{rlclcrcl}
&\fg_{0}\otimes V&\longrightarrow& \frkh,&&x\trr (u+m)&\triangleq&\sigma_0(x)\cdot_{\hg} (u+m),\\
&V\otimes\fg_{0}&\longrightarrow& \frkh, &&(u+m)\trl x&\triangleq&(u+m)\cdot_{\hg} \sigma_0(x),\\
&\fg_{{1}}\otimes V_0&\longrightarrow&\frkh_1,&&a\trr u&\triangleq&\sigma_1(a)\cdot_{\hg}u,\\
&V_0\otimes \fg_{{1}}&\longrightarrow&\frkh_1,&&u\trl a &\triangleq&u \cdot_{\hg}\sigma_1(a),\\
&\fg_{0}\otimes\fg_{0}\otimes V_0&\longrightarrow&\frkh_1,&& (x,y)\trr u&\triangleq&\widehat{l_{3}}(\sigma_0(x),\sigma_0(y),u),\\
&\fg_{0}\otimes V_0 \otimes\fg_{0}&\longrightarrow&\frkh_1, &&x\trr u \trl y&\triangleq&\widehat{l_{3}}(\sigma_0(x),u,\sigma_0(y)),\\
&V_{0}\otimes \fg_0\otimes\fg_{0}&\longrightarrow&\frkh_1, && u\trl (x,y)&\triangleq&\widehat{l_{3}}(u,\sigma_0(x),\sigma_0(y)),
\end{array}\right.
\end{equation}
for any $x,y\in\fg_{0}$, $a\in\fg_{{1}}$,
$u\in\frkh_{0}$ and $m\in\frkh_{{1}}$.

\begin{pro}\label{pro:2-modules}
With the above notations, $V$ is a representation of associative 2-algebra $\g$. Furthermore,  this representation  does not depend on the choice of the splitting $\sigma$. Moreover,  equivalent abelian extensions give the same representation of $\g$ on $\h$.
\end{pro}

\pf
 By the fact that $p$ is a strict homomorphism, it is easy to show that the  representation defined in \eqref{abelian:rep} are well-defined.

For direct computations, we have
 \begin{eqnarray*}
   &&(x\cg y)\trr (u+m)-x\trr (y\trr (u+m))\\
   &=&\sigma(x\cg y)\cdot_{\hg}(u+m)-\sigma(x)\cdot_{\hg}(\sigma(y)\cdot_{\hg}(u+m))\\
   &=&(\sigma(x)\cdot_{\hg}\sigma(y))\cdot_{\hg}(u+m)-\sigma(x)\cdot_{\hg}(\sigma(y)\cdot_{\hg}(u+m))\\
   &=&\dM_{\frkh}\widehat{l_3}(\sigma(x),\sigma(y),u)+\widehat{l_3}(\sigma(x),\sigma(y),\dM_{\frkh}m)\\
    &=&\dM_{\frkh}((x,y)\trr u)+(x,y)\trr \dM_{\frkh}m
 \end{eqnarray*}
which implies that
 \begin{eqnarray*}
  &&(x\cg y)\trr u-x\trr (y\trr u)=\dM_{\frkh}((x,y)\trr u),\\
  &&(x\cg y)\trr m-x\trr (y\trr m)=(x,y)\trr \dM_{\frkh}m.
 \end{eqnarray*}
Similarly, we have
 \begin{eqnarray*}
  &&( x\cg a)\trr u-x\trr(a\trr  u)=(x,\dM_\g (a))\trr u,\\
  &&( a\cg x)\trr u-a\trr(x\trr  u)=(\dM_\g (a),x)\trr u.
 \end{eqnarray*}
Next, we have
\begin{eqnarray*}
   &&(x\trr (u+m))\trl y-x\trr ((u+m)\trl y)\\
   &=&(\sigma(x)\chg (u+m)) \cdot_{\hg} \sigma(y)-\sigma(x)\cdot_{\hg}((u+m)\cdot_{\hg}\sigma(y))\\
   &=&\dM_{\frkh}\widehat{l_3}(\sigma(x),u,\sigma(y))+\widehat{l_3}(\sigma(x),\dM_{\frkh}m,\sigma(y))\\
   &=&\dM_{\frkh}(x\trr u\trl y)+x\trr (\dM_{\frkh}m) \trl y
\end{eqnarray*}
which implies that
 \begin{eqnarray*}
  &&(x\trr u)\trl y-x\trr (u\trl y)=\dM_{\frkh}(x\trr u\trl y),\\
&&(x\trr m)\trl y-x\trr (m\trl y)=x\trr (\dM_{\frkh}m) \trl y.
\end{eqnarray*}
Similarly, we have
 \begin{eqnarray*}
  &&(x\trr u)\trl a-x\trr (u\trl a)=x\trr u\trl \dM_{\frkh}(a),\\
  &&(a\trr u)\trl y-a\trr (u\trl y)= \dM_{\frkh}(a) \trr u\trl y.
\end{eqnarray*}
Also, we have
\begin{eqnarray*}
   &&   (u+m)\trl (x\cg y)-((u+m)\trl x)\trl y\\
   &=&(u+m) \cdot_{\hg}\sigma(x\cg y)-((u+m)\cdot_{\hg} \sigma(x))\cdot_{\hg}\sigma(y)\\
   &=&(u+m) \cdot_{\hg}(\sigma(x)\cdot_{\hg}\sigma(y))-((u+m)\cdot_{\hg} \sigma(x))\cdot_{\hg}\sigma(y)\\
   &=&-\dM_{\frkh}\widehat{l_3}(u,\sigma(x),\sigma(y))-\widehat{l_3}(\dM_{\frkh}m,\sigma(x),\sigma(y)),
\end{eqnarray*}
which implies that
 \begin{eqnarray*}
  &&u\trl (x\cg y)-(u\trl x)\trl y=\dM_{\frkh}(u\trl (x,y)),\\
  &&m\trl (x\cg y)-(m\trl x)\trl y=\dM_{\frkh}m\trl (x,y).
 \end{eqnarray*}
Similarly, we have
 \begin{eqnarray*}
  &&u\trl (x\cg a)-(u\trl x)\trl a=u\trl (x,\dM_{\fg} a),\\
    &&u\trl (a\cg x)-(u\trl a)\trl x=u\trl (\dM_{\fg} a, x).
 \end{eqnarray*}
By
\begin{eqnarray*}
&&\si (x)\chg\widehat{l}_3(\si (y),\si (z),u)+\widehat{l}_3(\si (x),\si (y),\si (z))\chg u \\
&=&\widehat{l}_3(\si (x)\cdot \si (y),\si (z),u)-\widehat{l}_3(\si (x),\si (y\cdot z),u)+\widehat{l}_3(\si (x),\si (y), \si (z)\cdot u),
\end{eqnarray*}
we have
\begin{eqnarray*}
x\trr((y,z)\trr u)+l_3^\g (x,y,z)\trr u =(x\cdot y,z)\trr u -(x, y\cdot z)\trr u+ (x,y)\trr (z\trr u).
\end{eqnarray*}
By
\begin{eqnarray*}
&&\si (x)\chg \widehat{l}_3(u,\si (y),\si (z))+\widehat{l}_3(\si (x),u,\si (y))\chg\si (z)\\
&=&\widehat{l}_3(\si (x)\cdot u,\si (y),\si (z))-\widehat{l}_3(\si (x), u\cdot \si (y),\si (z))+\widehat{l}_3(\si (x),u, \si (y)\cdot\si (z)),
\end{eqnarray*}
we have
\begin{eqnarray*}
&&x\trr (u  \trl (y\cdot z)) +(x \trr  u\trl y) \trl z= (x\trr u)  \trl (y, z)-x\trr (u\trl y)\trl z+x\trr u\trl (y\cdot z).
\end{eqnarray*}
By
\begin{eqnarray*}
&&\si (x)\chg\widehat{l}_3(\si (y),u,\si (z))+\widehat{l}_3(\si (x),\si (y),u)\chg\si (z)\\
&=&\widehat{l}_3(\si (x)\chg\si (y),u,\si (z))-\widehat{l}_3(\si (x), \si (y)\chg u,\si (z))+\widehat{l}_3(\si (x),\si (y), u\chg\si (z)),
\end{eqnarray*}
we have
\begin{eqnarray*}
 x\trr (y\trr u\trl z) +((x,y)\trr u)\trl z =(x\cdot y)\trr u\trl z  -x\trr (y\trr u)\trl z + (x,y)\trr (u\trl z).
\end{eqnarray*}
By
\begin{eqnarray*}
&&u\chg\widehat{l}_3(\si (x),\si (y),\si (z))+\widehat{l}_3(u,\si (x),\si (y))\chg\si (z)\\
&=&\widehat{l}_3(u\chg\si (x),\si (y),\si (z))-\widehat{l}_3(u, \si (x)\chg\si (y),\si (z))+\widehat{l}_3(u,\si (x), \si (y)\chg\si (z)),
\end{eqnarray*}
we have
\begin{eqnarray*}
u\trl l_3(x,y,z) + (u\trl (x,y))\trl z= (u\trl x) \trl (y,z) -u\trl (x\cdot y,z)+u\trl (x, y\cdot z).
 \end{eqnarray*}
 Thus we obtain a representation of associative 2-algebra.

Since $\h$ is abelian, we can show that the representations are independent of the choice of $\sigma$.
In fact, if we choose another splitting $\sigma':\g\to\hg$, then $p_0(\sigma(x)-\sigma'(x))=x-x=0$,
$p_1(\sigma(a)-\sigma'(a))=a-a=0$, i.e. $\sigma(x)-\sigma'(x)\in \Ker p_0=\h_0$,
$\sigma(a)-\sigma'(a)\in \Ker p_1=\h_{{1}}$.
Thus, $(\sigma(x)-\sigma'(x))\chg(u+m)=0$, $(\sigma(a)-\sigma'(a))\chg u=0$, which implies that  the representations $\trr$ are independent on the choice of $\sigma$. Similarly, $\trl$ are independent on the choice of $\sigma$ by the same reason.
We also have $\widehat{l_{3}}(\cdot,u,v)=0$, $\widehat{l_{3}}(u,\cdot,v)=0$, $\widehat{l_{3}}(u,v,\cdot)=0$ for any $u,v\in\frkh_0$, which implies that
$$\widehat{l_{3}}(\si' (x),\si' (y),u)=\widehat{l_{3}}(\si (x),\si (y),u), \quad \widehat{l_{3}}(u,\si' (x),\si' (y))=\widehat{l_{3}}(u,\si (x),\si (y)),$$
 $$\widehat{l_{3}}(\si' (x),u,\si' (y))=\widehat{l_{3}}(\si (x),u,\si (y)).$$
Therefore the representations are also independent on the choice of $\sigma$.

Suppose that $\E_{\widehat{\g}}$ and $\E_{\tilde{\g}}$ are equivalent abelian extensions, and $F:\widehat{\g}\longrightarrow\tilde{\g}$ is the associative 2-algebra homomorphism satisfying $F\circ i=j$, $q\circ
 F=p$ and $F_2(i(u),\alpha)=0$, for any
 $u\in\frkh_0,~\alpha\in\widehat{\g}_0$.
Choose linear sections $\sigma$ and $\sigma'$ of $p$ and $q$. Then we have $qF_0\si (x)=p\si (x)=x=q\si (x)$,
so $F_0\si (x)-\si'(x)\in \Ker q=\h_0$. Thus, we have
$$\si'(x)\chg (u+m)=F_0\si (x)\chg (u+m)=F_0(\si (x)\chg (u+m))=\si (x)\chg (u+m),$$
which implies that equivalent abelian extensions give the same $\rho^L_0$. Similarly, we can show that equivalent abelian extensions also give the same $\rho^R_0$, $\rho^L_1$, and $\rho^L_1$. At last, by the fact that $F=(F_0,F_1,F_2)$ is a homomorphism and $F_2(i(u),\cdot)=0$, we have
$$
\widehat{l_3}(\si (x),\sigma(y),u)=\widetilde{l_3}(F_0\sigma(x),F_0\sigma(y),u)=\widetilde{l_3}(\sigma'(x),\sigma'(y),u).
$$
Therefore, equivalent abelian extensions also give the same representations.
The proof is finished.
\qed\vspace{3mm}

Let $\sigma:\frkg\longrightarrow\widehat{\frkg}$  be a
splitting of the abelian extension \eqref{eq:ext1}. Define the following linear maps:
$$
\begin{array}{rlclcrcl}
\psi:&\frkg_{{1}}&\longrightarrow&\h_{0},&& \psi(a)&\triangleq&\widehat{\dM}\si(a)-\si(\dM_\frkg (a)),\\
\omega:&\otimes^2\frkg_{0}&\longrightarrow&\h_{0},&& \omega(x,y)&\triangleq&\si(x)\cdot_{\widehat{\frkg}}\si(y)-\si(x\cdot_\fg y),\\
\mu:&\frkg_{0}\otimes\frkg_{{1}}&\longrightarrow&\h_{{1}},&& \mu(x,a)&\triangleq&\si(x)\cdot_{\widehat{\frkg}}\si(a)-\si(x\cdot_\frkg a),\\
\nu:&\frkg_{{1}}\otimes\frkg_{0}&\longrightarrow&\h_{{1}},&& \nu(a,x)&\triangleq&\si(a)\cdot_{\widehat{\frkg}}\si(x)-\si(a\cdot_\fg x),\\
\theta:&\otimes^3\frkg_{0}&\longrightarrow&\h_{{1}},&&
\theta(x,y,z)&\triangleq&\widehat{l}_{3}(\si(x),\si(y),\si(z))-\si(l_{3}^\frkg(x,y,z)),
\end{array}
$$
for any $x,y,z\in\frkg_{0}$, $a\in\frkg_{{1}}$.

\begin{thm}\label{thm:2-cocylce}
Let $\E_{\hg}$ is an abelian extension of $\g$ by $\h$ given by \eqref{eq:ext1}, then $(\psi,\omega,\mu,\nu,\theta)$ is a $2$-cocycle of $\g$ with coefficients in $\frkh$.
\end{thm}

\pf
By the equality $\widehat{\dM}(\si (x)\cdot_{\widehat{\g}}\si (a))=\si (x)\cdot_{\widehat{\g}}\widehat{\dM}\si (a)$, we have
\begin{eqnarray*}
\dM_\frkh\mu(x,a)+\widehat{\dM}\si(x\cdot_\frkg a)&=&\si (x)\cdot_{\widehat{\g}}(\psi(a)+\si \dM_\frkg (a))\\
\dM_\frkh\mu(x,a)+\psi(x\cdot_\frkg a)+\si(x\cdot_\frkg\dM_\frkg (a))&=&\si (x)\chg\psi(a)+\si (x)\chg\si \dM_\frkg (a)
\end{eqnarray*}
Thus we  obtain that
\begin{eqnarray}\label{eq:c0}
x\trr \psi(a)+\omega(x,\dM_\fg a)=\psi(x\cdot_\fg a)+\dM_\frkh\mu(x,a).
\end{eqnarray}

Similarly, by the equality $\widehat{\dM}(\si (a)\chg \si (x))=\widehat{\dM}\si (a)\chg\si (x)$, we get
\begin{eqnarray*}
\dM_\frkh\nu(a,x)+\widehat{\dM}\si(a\cdot_\fg x)&=&(\psi(a)+\si \dM_\frkg (a))\chg \si (x)\\
\dM_\frkh\nu(a,x)+\psi(a\cdot_\fg x)+\si (\dM_\frkg (a)\cg x)&=&\psi(a)\chg\si (x)+\si \dM_\frkg (a)\chg\si (x).
\end{eqnarray*}
Thus we  obtain that
\begin{eqnarray}\label{eq:c1}
\psi(a)\trl x+\omega(\dM_\frkg (a),x)=\psi(a\cdot_\fg x)+\dM_\frkh\nu(a,x).
\end{eqnarray}

By the equality $\widehat{\dM}\si (a)\chg \si (b)=\si (a)\chg \widehat{\dM}\si (b)$, we obtain that
$$(\psi(a)+\si \dM_\frkg (a))\chg\si (b)=\si (a)\chg(\psi(b)+\si \dM_\frkg (b))$$
$$\psi(a)\chg\si (b)+\si \dM_\frkg (a)\chg \si (b)=\si (a)\chg \si \dM_\frkg (b)+ \si (a)\chg\psi(b)$$
\begin{eqnarray}
\psi(a)\trl b+\nu(\dM_\g (a),b)+\si(\dM_\g (a)\cg b)=a\trr (\psi(b))+\nu(a,\dM_\g (b))+\si(a\cg\dM_\g (b)).
\end{eqnarray}
\begin{eqnarray}\label{eqn:c2}
\psi(a)\trl b+\nu(\dM_\g (a),b)=a\trr \psi(b)+\nu(a,\dM_\g (b)).
\end{eqnarray}

By the equality
$$
(\si (x)\chg \si (y))\chg\si (z)-\si (x)\chg (\si (y)\chg\si (z))
=\widehat{\dM}\widehat{l_{3}}(\si (x),\si (y),\si (z)),
$$
we get
\begin{eqnarray}\label{eq:c3}
x\trr \omega(y,z)-\omega(x,y)\trl z+\omega(x, y\cdot_\fg z)-\omega(x\cdot_\fg y ,z)=\dM_\frkh\theta(x,y,z)+\psi(l_3^\g (x,y,z)).
\end{eqnarray}

We also have the equality
\begin{eqnarray*}
(\si (x)\chg\si (y))\chg\si (a)-\si (x)\chg(\si (y)\chg\si (a))=\widehat{l_{3}}(\si (x),\si (y),\widehat{\dM}(\si (a))).
\end{eqnarray*}
Consider the left hand side, we have
\begin{eqnarray*}
&&(\si (x)\chg\si (y))\chg\si (a)-\si (x)\chg(\si (y)\chg\si (a))\\
&=&(\sigma (x\cg y)+\omega(x,y))\chg\si (a)-\si (x)\chg \sigma(y\cg a+\mu(y,a))\\
&=&\sigma((x\cg y)\cg a)-\mu(x\cg y,a)-\omega(x,y)\trl a\\
&&-\sigma(x\cg (y\cg a))+\mu(x,(y\cg a))+x\trr \mu(y,a)\\
&=&\sigma l_3^\g (x,y,\dM_\g (a))-x\trr \mu(y,a)+\omega(x,y)\trl a\\
&&-\mu(x,y\cg a)+\mu(x\cg y,a).
\end{eqnarray*}
Consider the right hand side, we have
\begin{eqnarray*}
\widehat{l_{3}}(\si (x),\si (y),\widehat{\dM}(\si (a)))
&=&\widehat{l_{3}}(\si (x),\si (y),\sigma(\dM_\g (a))+\psi(a))\\
&=&\widehat{l_{3}}(\si (x),\si (y),\sigma(\dM_\g (a)))+\widehat{l_{3}}(\si (x),\si (y),\psi(a))\\
&=&\sigma l_3^\g (x,y,\dM_\g (a))+\theta(x,y,\dM_\g (a))-(x,y)\trr \psi(a).
\end{eqnarray*}
Thus, we have
\begin{eqnarray}
\nonumber&&\omega(x,y)\trl a-x\trr \mu(y,a)+\mu(x\cg y,a)-\mu(x,y\cg a)\\
&=&\theta(x,y,\dM_\g (a))-(x,y)\trr \psi(a).\label{eq:c4}
\end{eqnarray}

Similarly, by the equality
\begin{eqnarray*}
(\si (x)\chg\si (a))\chg\si (y)-\si (x)\chg(\si (a)\chg\si (y))=\widehat{l_{3}}(\si (x),\widehat{\dM}(\si (a)),\si (y)).
\end{eqnarray*}
we have
\begin{eqnarray*}
&&(\si (x)\chg\si (a))\chg\si (y)-\si (x)\chg(\si (a)\chg\si (y))\\
&=&(\sigma(x\cg a)+\mu(x,a))\chg\si (y)-\si (x)\chg\sigma(a\cg y)+\nu(a,y))\\
&=&\sigma(x\cg a)\cg y+\nu( x\cg a,y)-\mu(x,a)\trl y\\
&&-\sigma (x\cg (a\cg y))-\mu(x,(a\cg y))+x\trr \nu(a,y)\\
&=&\sigma l_3^\g (x,\dM_\g (a), y)-x\trr \nu(a,y)+\omega(x,a)\trl y\\
&&-\mu(x,(a\cg y))+\nu( x\cg a,y).
\end{eqnarray*}
\begin{eqnarray*}
\widehat{l_{3}}(\si (x),\widehat{\dM}(\si (a)),\si (y))
&=&\widehat{l_{3}}(\si (x),\sigma(\dM_\g (a))+\psi(a),\si (y))\\
&=&\widehat{l_{3}}(\si (x),\sigma(\dM_\g (a)),\si (y))+\widehat{l_{3}}(\si (x),\psi(a),\si (y))\\
&=&\sigma l_3^\g (x,\dM_\g (a),y)+\theta(x,\dM_\g (a),y)+x\trr\psi(a)\trl y.
\end{eqnarray*}
Thus, we have
\begin{eqnarray}
&&\omega(x,a)\trl y-x\trr \nu(a,y)+\nu( x\cg a,y)-\mu(x,a\cg y)=\theta(x,\dM_\g (a),y)+x\trr\psi(a)\trl y.\label{eq:c5}
\end{eqnarray}
Also, by definition we have
\begin{eqnarray*}
&&(\si (a)\chg\si (x))\chg\si (y)-\si (a)\chg(\si (x)\chg\si (y))\\
&=&(\sigma(a\cg x)+\nu(a,x))\chg\si (y)-\si (a)\chg(\sigma(x\cg y)+\omega(x,y))\\
&=&\sigma((a\cg x)\cg y)-\nu(a\cg x,y)-\nu(a,x)\trl y\\
&&-\sigma(a\cg (x\cg y))+\nu(a,x\cg y)+a\trr \omega(x,y)\\
&=&\sigma l_3^\g (a,\dM_\g x, y)-a\trr \omega(x,y)+\omega(a,x)\trl y\\
&&-\nu(a,(x\cg y))-\nu(a\cg x,y)+\mu(x,(a\cg y)).
\end{eqnarray*}
\begin{eqnarray*}
\widehat{l_{3}}(\widehat{\dM}(\si (a)),\si (x),\si (y))
&=&\widehat{l_{3}}(\sigma(\dM_\g (a))+\psi(a),\si (x),\si (y))\\
&=&\widehat{l_{3}}(\sigma(\dM_\g (a)),\si (x),\si (y))+\widehat{l_{3}}(\psi(a),\si (x),\si (y))\\
&=&\sigma l_3^\g (\dM_\g (a),x,y)+\theta(\dM_\g (a),x,y)+\psi(a)\trl (x,y).
\end{eqnarray*}
Thus, we have
\begin{eqnarray}
&&\omega(a,x)\trl y-a\trr \omega(x,y)+\nu(a\cg x,y)-\nu(a, x\cg y)=\theta(\dM_\g (a),x,y)-\psi(a)\trl(x,y). \label{eq:c6}
\end{eqnarray}

At last, by the Stasheff identity, we have
\begin{eqnarray*}
&&\si (x)\chg l_3(\si (y),\si (z),\si (t))+l_3(\si (x),\si (y),\si (z))\chg\si (t)\\
&=&\widehat{l_{3}}(\si (x)\chg\si (y),\si (z),\si (t))-\widehat{l_{3}}(\si (x),\si (y)\chg\si (z),\si (t))
+\widehat{l_{3}}(\si (x),\si (y),\si (z)\chg\si (t)).
\end{eqnarray*}
The left hand side is equal to
\begin{eqnarray*}
&&\si (x)\chg (\sigma l_3^\g (y,z,t)-\theta(y,z,t))+(\sigma l_3^\g (x,y,z)-\theta(x,y,z),\si (t))\chg\si (t)\\
&=&\sigma(x\cg l_3^\g (y,z,t))-\mu(x,l_3^\g (y,z,t))-x\trr \theta(y,z,t)\\
&&+\sigma(t\cg l_3^\g (x,y,z))-\nu(l_3^\g (x,y,z),t)-\theta(x,y,z)\trl t
\end{eqnarray*}
and the right hand side is equal to
\begin{eqnarray*}
&&\widehat{l_3}(\si(x\cg y)-\omega(x,y),\si (z),\si (t))+c.p\\
&=&\si l_3^\g (x\cg y,z,t)-\theta(x\cg y,z,t)-\omega(x,y)\trl (z,t)\\
&&-\si l_3^\g (x,y\cg z,t)+\theta(x,y\cg z,t)+x\trr \omega(y,z)\trl t\\
&&+\si l_3^\g (x,y, z\cg t)-\theta(x,y,z\cg t)-(x,y)\trr \omega(z,t)
\end{eqnarray*}
Thus, we obtain
\begin{eqnarray}
\nonumber&&x\trr \theta(y,z,t)+\theta(x,y,x)\trl t\\
\nonumber&&-\theta(x\cg y,z,t)+\theta(x,y\cg z,t)-\theta(x,y,z\cg  t)\\
\nonumber&=&\mu(x,l_3^\g (y,z,t))+\nu(l_3^\g (x,y,z),t)\\
\label{eq:c7}&&-\omega(x,y)\trl (z,t)+x\trr \omega(y,z)\trl t-(x,y)\trr \omega(z,t).
\end{eqnarray}
By \eqref{eq:c0}-\eqref{eq:c7}, we deduce that
$(\psi,\omega,\mu,\nu,\theta)$ is a 2-cocycle.\qed\vspace{3mm}

Now we can construct an associative 2-algebra structure  on $\g\oplus\frkh$ using the 2-cocycle given above. More precisely, we have
\begin{equation}\label{eq:multiplication}
\left\{\begin{array}{rcl}
\dM_{\gh}(a+m)&\triangleq&\dM_\g (a)+\psi(a)+\dM_\frkh(a),\\
{(x+u)\cdot_{\gh} (y+v)}&\triangleq&x\cg y+\omega(x,y)+x\trr v+u\trl y,\\
{(x+u)\cdot_{\gh} (a+m)}&\triangleq& x\cg a+\mu(x,a)+x\trr m+u\trl a,\\
{(a+m)\cdot_{\gh} (x+u)}&\triangleq&a\cg x+\nu(a,x)+a\trr u+m\trl x,\\
l_3^{\gh}(x+u,y+v,z+w)&\triangleq& l_3^\g (x,y,z)+\theta(x,y,z)\\
&&+(x,y)\trr w+x\trr v\trl z+u\trl (y,z),
\end{array}\right.
\end{equation}
for any $x,y,z\in\fg_{0}$, $a\in\fg_{{1}}$,
$u\in\frkh_{0}$ and $m\in\frkh_{{1}}$. Thus any extension $E_{\widehat{\g}}$ given by \eqref{eq:ext1} is
isomorphic to
\begin{equation}\label{ext2}
\CD
  0 @>0>>  \frkh_{{1}} @>i_1>> \g_{{1}}\oplus \frkh_{{1}} @>p_1>> \g_{{1}} @>0>> 0 \\
  @V 0 VV @V \dM_\frkh VV @V \widehat{\dM} VV @V\dM_\g VV @V0VV  \\
  0 @>0>> \frkh_{0} @>i_0>> \g_0\oplus \frkh_0 @>p_0>> \g_0@>0>>0,
\endCD
\end{equation}
where the associative 2-algebra structure on $\frkg\oplus\frkh$ is given by
\eqref{eq:multiplication}, $(i_0,i_1)$ is the inclusion and $(p_0,p_1)$ is the projection. We
denote the extension \eqref{ext2} by $\E_{\frkg\oplus\frkh}$.

In the sequel, we study the relation between equivalent classes of abelian extensions and the second cohomology group $\mathbf{H}^2(\frkg;{V})$.

\begin{thm}\label{mainthm44}
There is a one-to-one correspondence between equivalence classes of abelian extensions of the associative 2-algebras $\frkg$ by $\h$ and the second cohomology group $\mathbf{H}^2(\frkg;{V})$.
\end{thm}

\pf
Let $\E'_{\gh}$ be another abelian extension determined by the 2-cocycle $(\psi',\omega',\mu',\nu',\theta')$. Denote the corresponding associative 2-algebra structure on $\g\oplus \h$ by $(\dM',\cdot',l_3')$. We only need to show that $\E_{\gh}$ and $\E'_{\gh}$ are equivalent if and only if 2-cocycles  $(\psi,\omega,\mu,\nu,\theta)$ and $(\psi',\omega',\mu',\nu',\theta')$ are in the same cohomology class.

If $\E_{\gh}$ and $\E'_{\gh}$ are equivalent, let
$F=(F_0,F_1,F_2):\E_{\gh}\longrightarrow\E'_{\gh}$
be the corresponding homomorphism.

 Since $F$ is an equivalence of extensions, we have
$$F_2(u,v)=0,\quad F_2(x,u)=0,\quad F_2(x,y)\in\frkh_{{1}},$$
and there exist two linear maps $\lambda_0:\g_0\longrightarrow\frkh_0$
and $\lambda_1:\g_{{1}}\longrightarrow\frkh_{{1}}$ such that
 $$F_0(x+u)=x+\lambda_0(x)+u,\quad
F_1(a+m)=a+\lambda_1(a)+m.$$ Set  $\lambda_2=F_2|_{\otimes^2\g_0}$.

First, by the equality
\begin{eqnarray*}
\label{eqn:fi} \dM'F_1(a)&=&F_0\dM_{\gh}(a),
\end{eqnarray*}
$$\dM'(a+\lambda_1(a))=\dM_\g (a)+\lambda_0(\dM_\g (a))+\psi(a)$$
we have
\begin{equation}\label{eq:exact1}
\psi(a)-\psi'(a)=\dM_\frkh \lambda_1(a)-\lambda_0(\dM_\g (a)).
\end{equation}

 Furthermore, we have
$$F_0(x\cdot_{\gh} y)-F_0(x)\cdot' F_0(y)=\dM'F_2(x,y),$$
$$(x\cdot y)+\omega(x,y)+\lambda_0(x\cdot y)-(x+\lambda_0(x))\cdot' (y+\lambda_0(y))=\dM_\frkh\circ \lambda_2(x,y),$$
$$(x\cdot y)+\omega(x,y)+\lambda_0(x\cdot y)-(x\cdot y)-\omega'(x,y)-x\trr \lambda_0(y)-\rho^R_0(y)\lambda_0(x)=\dM_\frkh\circ \lambda_2(x,y),$$
which implies that
\begin{equation}\label{eq:exact2}
\omega(x,y)-\omega'(x,y)
=x\trr \lambda_0(y)-\lambda_0(x)\trl y-\lambda_0(x\cg y)+\dM_\frkh\circ
\lambda_2(x,y).
\end{equation}

Similarly, by $F_1(x\cdot a)-F_0(x)\cdot' F_1(a)=F_2(x,\widehat{\dM}a)$,
we get
\begin{equation}\label{eq:exact3}
\mu(x,a)-\mu'(x,a)=x\trr \lambda_1(a)-\lambda_0(x)\trl a-\lambda_1(x\cg a)+\lambda_2(x,\dM_\g (a)),
\end{equation}
and by $F_1(a\cdot x)-F_1(a)\cdot' F_0(x)=F_2(\widehat{\dM}a,x)$, we get
\begin{equation}\label{eq:exact32}
\nu(a,x)-\nu'(a,x)=a\trr \lambda_1(x)-\lambda_0(a)\trl x-\lambda_1(a\cg x)+\lambda_2(\dM_\g (a),x).
\end{equation}

At last, by the equality
\begin{eqnarray*}
&&F_1l_3^{\gh}(x,y,z)-l_3'(F_0(x),F_0(y),F_0(z))\\
&=&F_2(x, y\cg z)- F_2(x\cg y,z) + F_0(x)\cdot_{\g'} F_2(y,z) - F_2(x,y)\cdot_{\g'} F_0(z).
\end{eqnarray*}
the left hand side is equal to
\begin{eqnarray*}
&&F_1(l_3^{\g}(x,y,z)+\theta(x,y,z))-l_3'(x+\lambda_0(x),y+\lambda_0(y),z+\lambda_0(z))\\
&=&l_3^{\g}(x,y,z)+\theta(x,y,z)+\lambda_1(l_3^\g (x,y,z))\\
&&-\{l_3^{\g}(x,y,z)+\theta'(x,y,z)-(x,y)\trr \lambda_0(z)-x\trr\lambda_0(y)\trl z-\lambda_0(x)\trl(y,z)\},
\end{eqnarray*}
the right hand side is equal to
\begin{eqnarray*}
&&\lambda_2(x, y\cdot z)- \lambda_2(x\cdot y,z) \\
&&+(x+\lambda_0(x))\cdot' \lambda_2(y,z)- \lambda_2(x,y)\cdot' (z+\lambda_0(z))\\
&=&\lambda_2(x, y\cdot z+\omega(y,z))- \lambda_2(x\cdot y+\omega(x,y),z)\\
&&+x\trr \lambda_2(y,z)-\lambda_2(x,y)\trl z)\\
&=&\lambda_2(x, y\cdot z)- \lambda_2(x\cdot y,z)\\
&&+x\trr \lambda_2(y,z)-\lambda_2(x,y)\trl z.
\end{eqnarray*}
Thus, we have
\begin{eqnarray}
\nonumber(\theta-\theta')(x,y,z)&=&x\trr \lambda_2(y,z)-\lambda_2(x,y)\trl z+\lambda_2(x, y\cdot z)- \lambda_2(x\cdot y,z)\\
\nonumber&&-\lambda_1l_3^\g (x,y,z)+(x,y)\trr \lambda_0(z)\\
\label{eq:exact4}&&+x\trr \lambda_0(y)\trl z+\lambda_0(x)\trl (y,z)
\end{eqnarray}
By \eqref{eq:exact1}-\eqref{eq:exact4}, we deduce that $(\psi,\omega,\mu,\nu,\theta)-(\psi',\omega',\mu',\nu',\theta')=D(\lambda_0,\lambda_1,\lambda_2)$. Thus, they are in the same cohomology class.

Conversely, if  $(\psi,\omega,\mu,\nu,\theta)$ and $ (\psi',\omega',\mu',\nu',\theta')$ are in the same cohomology class, assume that $(\psi,\omega,\mu,\nu,\theta)-(\psi',\omega',\mu',\nu',\theta')=D(\lambda_0,\lambda_1,\lambda_2)$. Then define $(F_0,F_1,F_2)$ by
$$F_0(x+u)=x+\lambda_0(x)+u,\quad
F_1(a+m)=a+\lambda_1(a)+m,\quad F_2(x+u,y+v)=\lambda_2(x,y). $$
Similar as the above proof, we can show that $(F_0,F_1,F_2)$ is an equivalence. We omit the details.
\qed

\section{Applications}
We provide a special case of the representation and cohomology theory of associative 2-algebras as well as some applications.
A associative 2-algebra with $l_3=0$ is called strict associative 2-algebra.

%

This lead to the concept of crossed modules over associative algebras.

\begin{defi}
 A crossed module over associative algebra $\pp$ is a  triple $(\hh,\pp, f)$  where $\pp$  is an associative algebra,
$\hh$ is a representation of $\pp$ such that for any $x\in\pp$, $a, b\in\hh$, the map  ${f}:\hh\longrightarrow\pp$ satisfying the following conditions:
\begin{eqnarray}
 \label{crossed01} {f} (x\cdot {a})&=&x\cdot {f}({a}),\quad {f}({a}\cdot x)={f}({a})\cdot x,\\
  \label{crossed02}~{f}({a})\cdot {b}&=&{a}\cdot {{f}({b})}.
\end{eqnarray}
The map ${f}:\hh\longrightarrow\pp$ is called an equivariant map if it satisfying \eqref{crossed01}.
\end{defi}

We remark that our definition is different from the concept of  crossed module of associative algebras which require $(\hh,\cdot_\hh)$ to be an associative algebra, an action of  $\pp$ on $\hh$,  and a further condition
$${a}\cdot_\hh {b}={f}({a})\cdot {b}={a}\cdot {f}({b}).$$
In fact, we find these conditions are not necessary for our strict associative 2-algebras.

\begin{pro}\label{thm:correspondence}
There is a one-to-one correspondence between strict associative 2-algebras and crossed modules over associative algebras.
\end{pro}

\pf
We give a sketch of the proof.
Let $\fg_1\stackrel{\dM_\g}{\longrightarrow} \fg_0$ be a  strict associative 2-algebras, define $\hh=\fg_1$, $\pp=\fg_0$, and the following
 multiplication in $\pp$:
\begin{eqnarray*}
~x\cdot y&=&l_2^\g(x,y),\quad\forall~x,y\in A_0.
\end{eqnarray*}
By condition $(d)$, we obtain that $(\pp, \cdot_\pp)$ is an associative algebra.
Also by condition $(e_1)$--$(e_3)$,  we get  a representation of $\pp$ on $\hh$ by
$$
x\cdot a=l_2^\g(x,a),\quad a\cdot x=l_2^\g(a,x),\quad \forall~ x\in \pp, a\in\hh.
$$
Let $f=\dM$, from  conditions $(a), (b)$ and $(c)$,  we obtain a crossed module over associative algebra.

Conversely,  a crossed module over associative algebra $\pp$
 gives rise to a strict associative 2-algebra with $\dM_\g={f}$,
$\fg_1=\hh$ and $\fg_0=\pp$, where the multiplications are given by:
\begin{eqnarray*}
~ l_2^\g(x,y)&\triangleq&x\cdot y,\\
~l_2^\g(x,a)&\triangleq&x\cdot a,\qquad l_2^\g(a,x)\triangleq a\cdot x.
\end{eqnarray*}
Now the crossed module conditions  \eqref{crossed01} and  \eqref{crossed02} are exactly conditions $(a), (b)$ and $(c) $.
Thus  a crossed module over associative algebra give rise to a strict associative 2-algebra.
\qed

%

\begin{defi}
Let $(\hh,\pp, f)$ be a crossed module over associative algebra.
A representation of $(\hh,\pp,f)$ is an object $(V,W,\varphi)$   such that the following conditions are satisfied:

(i)  $V $ and $W $  are representations of  $\pp$  respectively;

(ii) $\varphi:V \to W $ is an equivariant map:
\begin{eqnarray}\label{deflm01}
\varphi(x\cdot v)=x\cdot  \varphi(v),\quad \varphi(v\cdot x)= \varphi(v)\cdot  x;
\end{eqnarray}

(iii) there exists linear maps $\trl: W\otimes \hh\to V$ and $\trr: \hh\otimes W\to V$ such that:
\begin{eqnarray}\label{deflm02}
\varphi (w\trl a)=w\cdot f(a),\quad\varphi (a\trr w)= f(a)\cdot w,
\end{eqnarray}
and
 \begin{eqnarray}\label{deflm03}
f(a)\cdot v'={a}\trr\varphi(v'), \quad \varphi(v)\trl {a}'=v\cdot f(m');
 \end{eqnarray}

(iv) these structure satisfying the following compatibility conditions
\begin{eqnarray}
\label{deflm11}
x \cdot (w\trl a)=(x\cdot w)\trl a+w \trl(x\cdot a),\\
\label{deflm12}
 w\trl  (x\cdot a)= (w\cdot x)\trl a+ x\cdot  (w\trl a),\\
 \label{deflm13}
 a\trr (w\cdot x)= (a\trr w)\cdot x + w\trl (a\cdot x),\\
\label{deflm14}
a\trr (x\cdot w)=(a\cdot x) \trr w+x\trr (a\trr w),\\
\label{deflm15}
x \cdot (a\trr w)=(x\cdot a)\trr w+a\trr (x\cdot w),\\
\label{deflm16}
 w\trl  (a\cdot x)= (w\trl a)\cdot x+ a\trr  (w\cdot x),
\end{eqnarray}
where $x\in \pp$, $a\in \hh$, $v\in V$ and  $w\in W$.
\end{defi}

For example, let $(V,W,\varphi)=(\hh,\pp, f)$, then we get the adjoint representation of $(\hh,\pp, f)$ on itself as follows:
$\hh$ and $\pp$ are representations of $\pp$ in a natural way,
 $\varphi=f$ is an equivariant map:
\begin{eqnarray}
f(x\cdot a)=x\cdot f(a),
\end{eqnarray}
and there exists  maps $\trr:\hh\otimes \pp\to \hh, {a}\trr x\triangleq {a}\cdot x$ and $\trl:\pp\otimes \hh\to \hh, x\trl {a}\triangleq x\cdot {a}$ such that
\begin{eqnarray}
f({a}\trr x)=f({a}\cdot x)=f({a})\cdot x.
\end{eqnarray}

We construct semidirect product of a crossed module over associative algebra $(\hh,\pp, f)$ with its representation $(V,W,\varphi)$.
\begin{pro}\label{pro:semidirectproduct}
  Given a representation of crossed module over associative algebra $(\hh,\pp, f)$ on $(V,W,\varphi)$. Define on $(\hh\oplus V,\pp\oplus W,\widehat{f}=f+\varphi)$
  the following maps
\begin{equation}
\left\{\begin{array}{rcl}
\widehat{f}({a}+v)&\triangleq&{f}(a)+\varphi(v),\\
{(x+w)\cdot (x'+w')}&\triangleq&x\cdot x'+x\cdot w'+w\cdot x',\\
{(x+w){\cdot} ({a}+v)}&\triangleq& x\cdot {a} +x\cdot v+ w\trl {a},\\
{({a}+v){\cdot} (x+w)}&\triangleq& {a}\cdot x +v\cdot x+ {a}\trr w,
\end{array}\right.
\end{equation}
for all $x, x'\in \pp$, ${a},{a}'\in \hh$, $v,v'\in V$ and  $w,w'\in W$.
 Then $(\hh\oplus V,\pp\oplus W, \widehat{f})$ is a crossed module over associative algebra, which is called the semidirect product of $(\hh,\pp, f)$ and $(V,W,\varphi)$.
\end{pro}

\begin{proof}
First, we verify that $\widehat{f}$ is an equivariant map. We have the equality
\begin{eqnarray}\label{semi01}
\widehat{f}{\big((x+w)\cdot ({a}+v)\big)}=(x+w)\cdot \widehat{f}({a}+v).
\end{eqnarray}
The left hand side of \eqref{semi01} is
\begin{eqnarray*}
\widehat{f}{\big((x+w)\cdot ({a}+v)\big)}&=&\widehat{f}{\big( x\cdot {a} +x\cdot v+ w\trl {a}\big)}\\
&=&f( x\cdot {a} )+\varphi(x\cdot v)+\varphi (w\trl {a}),
\end{eqnarray*}
and the right hand side of \eqref{semi01} is
\begin{eqnarray*}
{(x+w)\cdot \widehat{f}({a}+v)}&=&(x+w)\cdot  (\dM(a)+\varphi(v))\\
&=&x\cdot\dM(a)+x\cdot \varphi(v)+w\cdot f(a).
\end{eqnarray*}

Si{a}ilarly, we have the equality
\begin{eqnarray}\label{semi02}
\widehat{f}{\big(({a}+v)\cdot (x+w)\big)}=\widehat{f}({a}+v)\cdot (x+w).
\end{eqnarray}
The left hand side of \eqref{semi02} is
\begin{eqnarray*}
\widehat{f}{\big( ({a}+v)\cdot (x+w)\big)}&=&\widehat{f}{\big({a}\cdot x +v\cdot x+ {a}\trr w\big)}\\
&=&f({a}\cdot x)+\varphi(v\cdot x)+\varphi ({a}\trr w),
\end{eqnarray*}
and the right hand side of \eqref{semi02} is
\begin{eqnarray*}
{\widehat{f}({a}+v)\cdot (x+w)}&=&(\dM(a)+\varphi(v))\cdot (x+w)\\
&=&f(a)\cdot x+\varphi(v)\cdot x+f(a)\cdot w.
\end{eqnarray*}
Thus the two sides of  \eqref{semi01} and \eqref{semi02}  are equal to each other by conditions \eqref{deflm01} and \eqref{deflm02}.

Next, we verify that $\hh\oplus V$ carries a representation structure over associative algebra $\pp\oplus W$.
The left module condition is
\begin{eqnarray}\label{sembimod01}
\notag&& (x+w)\cdot\big((x'+w')\cdot({a}+v)\big)\\
&=&\big((x+w)(x'+w')\big)\cdot ({a}+v)+  (x'+w')\cdot\big((x+w)\cdot({a}+v)\big).
\end{eqnarray}
The left hand side of \eqref{sembimod01} is
\begin{eqnarray*}
&& (x+w)\cdot\big((x'+w')\cdot({a}+v)\big)\\
&=&(x+w)\cdot\big(x'\cdot {a}+x'\cdot v+w'\trl {a}\big)\\
&=&x\cdot (x'\cdot {a})+x\cdot (x'\cdot v)+x\cdot (w'\trl {a})+w\trr (x'\cdot w),
\end{eqnarray*}
and the right hand side of  \eqref{sembimod01} is
\begin{eqnarray*}
&&\big( (x+w)\cdot (x'+w')\big)\cdot ({a}+v)\\
&=&\big((x\cdot x')+x\cdot w'+w\cdot x'\big)\cdot({a}+v)\\
&=&(x\cdot x')\cdot {a}+(x\cdot x')\cdot v+(x\cdot w')\trr {a}+(w\cdot x')\trl {a},
\end{eqnarray*}
\begin{eqnarray*}
&& (x'+w')\cdot\big((x+w)\cdot({a}+v)\big)\\
&=&(x'+w')\cdot\big(x\cdot {a}+x\cdot v+w\trl {a}\big)\\
&=&x'\cdot (x\cdot {a})+x'\cdot (x\cdot v)+x'\cdot (w\trl {a})+w'\trr (x\cdot w).
\end{eqnarray*}
Thus the two side of \eqref{sembimod01} is equal to each other if and only if \eqref{deflm11} and \eqref{deflm12} hold.
Similar computations show that $\hh\oplus V$ carries a right module structure over $\pp\oplus W$.
 if and only if \eqref{deflm13} and \eqref{deflm14} hold, and it satisfy the representation conditions if and only if \eqref{deflm11}--\eqref{deflm16} hold.

Finally, we verify that
\begin{eqnarray}\label{bimod333}
 \widehat{f}({a}+v)\cdot ({a}'+v')=({a}+v)\cdot  \widehat{f}({a}'+v').
 \end{eqnarray}
 The left hand side of \eqref{bimod333} is
 \begin{eqnarray*}
 \widehat{f}({a}+v)\cdot ({a}'+v')&=&(f(a)+\varphi(v))\cdot ({a}'+v')\\
 &=&f(a)\cdot {a}'+f(a)\cdot v'+\varphi(v)\trl {a}',
 \end{eqnarray*}
 and the right hand side \eqref{bimod333} of
  \begin{eqnarray*}
({a}+v)\cdot  \widehat{f}({a}'+v')&=&({a}+v)\cdot (f({a}')+\varphi(v'))\\
 &=&{a}\cdot f({a}')+v\cdot f({a}')+{a}\trr\varphi(v').
 \end{eqnarray*}
 Thus   \eqref{bimod333}  holds by condition \eqref{deflm03}.
This complete the proof.
\end{proof}

Next we develop a cohomology theory for a crossed module over associative algebra $(\hh,\pp, f)$ with coefficients in a representation $(V,W,\varphi)$.
The $k$-cochain  $C^{k}((\hh,\pp,f),(V,W,\varphi))$ is defined  to be the space:
\begin{eqnarray}
\Hom(\otimes^{k} \pp,W)\oplus\Hom(\mathcal{A}^{k-1},V)\oplus \Hom(\mathcal{A}^{k-2},W)\oplus \Hom(\hh^{k-1},V).
\end{eqnarray}
where
\begin{eqnarray*}
 &&\Hom(\mathcal{A}^{k-1},V):=\bigoplus_{i=1}^k  \Hom(\underbrace{{\pp}\otimes\cdots\otimes {\pp}}_{i-1}\otimes \hh\otimes \underbrace{{\pp}\otimes\cdots\otimes {\pp}}_{k-i},V),
 \end{eqnarray*}
is the direct sum of all possible tensor powers of ${\pp}$ and $\hh$ in which ${\pp}$ appears $k-1$ times but $\hh$ appears only once  and similarly for
$\Hom(\mathcal{A}^{k-2},W).$

The cochain complex is given by
\begin{equation} \label{eq:complex}
\begin{split}
 & \qquad \qquad\  W\stackrel{D_0}{\longrightarrow}\\
 &  \quad\Hom(\pp,W)\oplus\Hom(\hh,V)\stackrel{D_1}{\longrightarrow}\\
 & \Hom({\pp}\otimes {\pp}, W)\oplus \Hom(\mathcal{A}^{1}V)\oplus \Hom(\hh,W)\stackrel{D_2}{\longrightarrow}\\
  & \Hom(\otimes^3 \pp, W)\oplus \Hom(\mathcal{A}^{2}, V)\oplus \Hom(\mathcal{A}^{1},W)\oplus \Hom(\otimes^2\hh,W)\stackrel{D_3}{\longrightarrow}\cdots,
\end{split}
\end{equation}

We define the following maps:
\begin{eqnarray*}
&&\varphi^\sharp: \Hom(\mathcal{A}^{k-1},V)\to \Hom(\mathcal{A}^{k-1},W),\\
&&h=(h_1, h_2): \Hom(\mathcal{A}^{k},W)\to \Hom(\mathcal{A}^{k},V),\\
&&l=(l_1, l_2): \Hom(\mathcal{A}^{k},W)\to \Hom(\mathcal{A}^{k-1},W)),\\
&&k=(k_1, k_2): \Hom(\mathcal{A}^{k-2},W)\to \Hom(\otimes^{k-1}\fg_1,W)),\\
\end{eqnarray*}
by
\begin{eqnarray*}
&&\varphi^\sharp(\mu)(x_1, \cdots, x_{k-1},{a})=\varphi(\mu(x_1, \cdots, x_{k-1},{a})),\\
&&\varphi^\sharp(\nu)({a},x_1, \cdots, x_{k-1})=\varphi(\nu({a},x_1, \cdots, x_{k-1})),\\
&&h_1(\omega)(x_1, \cdots, x_{k},{a})=\omega( x_1, \cdots, x_{k})\trl {a},\\
&&h_2(\omega)({a},x_1, \cdots, x_{k})={a}\trr \omega( x_1, \cdots, x_{k}),\\
&&l_1(\omega)({a}, x_1, \cdots, x_{k-1})=\omega(f(a),x_1, \cdots, x_{k-1}),\\
&&l_2(\omega)(x_1, \cdots, x_{k-1},{a})=\omega(x_1, \cdots, x_{k-1}, f(a)),\\
&&k_1(\mu)({a}, b_1, \cdots, b_{k-2})=\mu(f(a),b_1, \cdots, b_{k-2}),\\
&&k_2(\nu)(b_1, \cdots, b_{k-2},{a})=\nu(b_1, \cdots, b_{k-2}, f(a)).
\end{eqnarray*}
Thus, one can write the cochain complex in components as follows:
$$
\xymatrix{
  W \ar[d]_{-\delta_1}\ar[dr]^{\varphi\sharp}& & & & \\
  \Hom(\pp,W)\qquad\ar[d]_{-\delta_1}\ar@<.5em>[dr]^ {-h_1}\ar@<1em>[dr]_{-h_2} \ar@<.5em>[drr]^ {\qquad-l_1} \ar@<.1em>[drr]_{\quad\qquad\qquad-l_2} &\hspace{-1cm} \oplus  \Hom(\hh,V)  \ar[d]_{\delta_2}\ar[dr]^{\varphi^\sharp}&&& \\
\Hom({\pp}\otimes {\pp}, W)\qquad\ar[d]_{-\delta_1}\ar@<.5em>[dr]^ {-h_1}\ar@<1em>[dr]_{-h_2}
\ar@<.5em>[drr]^ {\qquad-l_1} \ar@<.1em>[drr]_{\quad\qquad\qquad-l_2} & \oplus \Hom(\mathcal{A}^1, V)
            \ar[d]_{\delta_2}\ar[dr]^{\varphi^\sharp} \ar@<.5em>[drr]^ {\qquad-k_1} \ar@<.1em>[drr]_{\quad\qquad\qquad-k_2}   & \hspace{-.5cm}\oplus\Hom(\hh,W)\ar[d]_{-\delta_3}\ar[dr]^{\varphi^\sharp}  &&\\
\Hom(\otimes^3 \pp, W) &\oplus \Hom(\mathcal{A}^2, V)         & \hspace{-.5cm} \oplus \Hom(\mathcal{A}^1,W)
   & \oplus\Hom(\otimes^2\hh,V)& &  \\
           }
$$
Therefore we get a cochain complex $C^k((\hh,\pp,f),(V,W,\varphi))$ whose cohomology group
$$H^k((\hh,\pp,f),(V,W,\varphi))=Z^k((\hh,\pp,f),(V,W,\varphi))/B^k((\hh,\pp,f),(V,W,\varphi))$$
is defined as the cohomology group of $(\hh,\pp,f)$ with coefficients in $(V,W,\varphi)$.
We list the low dimensional coboundary operator as follows.

For 1-cochain $(N_0,N_1)\in \Hom(\pp,W)\oplus\Hom(\hh,V)$, the coboundary map is
\begin{eqnarray*}
D_1(N_0,N_1)(a)
&=&\varphi\circ N_1(a)-N_0\circ f(a),\\
D_1(N_0,N_1)(x,y)
&=&N_0(x)\cdot y +  x\cdot N_0(y) - N_0(x\cdot y),\\
D_1(N_0,N_1)(x,{a})
&=&N_0(x)\cdot {a}+ x\cdot N_1(a) - N_1(x\cdot {a}),\\
D_1(N_0,N_1)({a},x)
&=&N_1(a)\cdot x+ {a}\cdot N_0(x) - N_1({a}\cdot x).
\end{eqnarray*}

For a 2-cochain
\begin{eqnarray*}
(\psi,\omega,\mu,\nu)&\in &\Hom({\pp}\otimes {\pp}, W)\oplus \Hom(\mathcal{A}^1,V)\oplus \Hom(\hh,W),
\end{eqnarray*}
 we get
\begin{eqnarray*}
D_2(\psi,\omega,\mu,\nu)&\in &\Hom(\otimes^3 \pp, W)\oplus \Hom(\mathcal{A}^2,V)\oplus \Hom(\mathcal{A}^1,W)\oplus \Hom(\otimes^2\hh,W),
\end{eqnarray*}
where
\begin{eqnarray*}
 &&\Hom(\mathcal{A}^2,V):=\Hom({\pp}\otimes {\pp}\otimes \hh,V)\oplus \Hom({\pp}\otimes \hh\otimes {\pp},V)\oplus \Hom(\hh\otimes {\pp}\otimes {\pp},V),\\
&& \Hom(\mathcal{A}^1,W):= \Hom({\pp}\otimes \hh,W)\oplus \Hom(\hh\otimes {\pp},W).
\end{eqnarray*}
Then a 2-cocycle is  a 2-cochain $(\psi,\omega,\mu,\nu)\in \Hom({\pp}\otimes {\pp}, W)\oplus \Hom(\pp,\Hom(\hh,V))\oplus \Hom(\hh,W)$ such that the following conditions hold:
\begin{eqnarray}
&&\psi ( x\cdot {a} )+\varphi (\mu (x,{a}))-x\cdot\psi (a)-\omega (x, f(a))=0,\\
&&\psi ( {a}\cdot x )+\varphi(\nu ({a},x))-\omega (f(a),x)-\psi (a)\cdot x=0,\\
&&\psi (a)\trl n+\mu({f} (a),n)-{a}\trr \psi (b)-\nu({a},{f} (b))=0,\\[1em]
\notag&&+\omega (x, y\cdot z)-\omega(x\cdot y,  z)-\omega ( y, x\cdot z)\\
&&x\cdot \omega (y, z)-\omega (x, y)\cdot  z- y\cdot \omega (x, z)=0,\\[1em]
\notag&&x\cdot\mu(y,{a})+\mu(x,y\cdot {a})-\omega(x,y)\trr {a}\\
&&-\mu(x\cdot y,{a})- y\cdot\mu(x,{a})-\mu( y, x\cdot {a})=0,\\[1em]
\notag&&{a}\trl\omega(x,y)+\nu({a}, (x\cdot y))-\nu({a},x)\cdot  y\\
&&-\nu({a}\cdot x, y)-x\cdot \nu({a},y)-\mu(x, {a}\cdot y)=0,\\[1em]
\notag&&x\cdot \nu({a},y))+\mu(x, {a}\cdot y)-\nu(x\cdot {a},  y)-\mu(x,{a})\cdot y\\
&&-{a}\trr\omega(x,y))-\nu( {a}, (x\cdot y))=0.
\end{eqnarray}
The reader can compare these conditions with conditions \eqref{eq:coc01}--\eqref{eq:coc07} in Section 2.

\subsection{Infinitesimal deformations}

Let $(\hh,\pp, f)$ be a crossed module over associative algebra  and $\psi: \hh \to\pp,~\omega: {\pp}\otimes {\pp}\to \pp$, $\mu:\pp \otimes \hh \to \hh, ~\nu: \hh\otimes \pp\to \hh$ be linear maps. Consider a $\lambda$-parametrized family of linear operations:
\begin{eqnarray*}
 {f}_\lam (a)&\triangleq&{f} (a)+\lambda\psi (a),\\
 {x\cdot_\lam y}&\triangleq& x\cdot y+ \lambda\omega (x, y),\\
 {x\cdot_\lam a}&\triangleq&  x\cdot {a} + \lambda\mu(x,{a}),\\
  {{a}\cdot_\lam x}&\triangleq&  {a}\cdot x + \lambda\nu({a}, x).
 \end{eqnarray*}

If  $(\hh_\lambda, \pp\dlam, f\dlam)$ forms a crossed module over associative algebra, then we say that
$(\psi,\omega,\mu,\nu)$ generates a 1-parameter infinitesimal deformation of $(\hh,\pp, f)$.

\begin{thm}\label{thm:deformation} With the notations above,
$(\psi,\omega,\mu,\nu)$ generates a $1$-parameter infinitesimal deformation of $(\hh,\pp, f)$  if and only if the following conditions hold:
\begin{itemize}
  \item[\rm(i)] $(\psi,\omega,\mu,\nu)$ is a $2$-cocycle of $(\hh,\pp,f)$ with coefficients in the adjoint representation;

  \item[\rm(ii)] $(\hh,\pp,\psi)$ is a crossed module over associative algebra with multiplication $\omega$ on $\pp$ and representation structures of $\pp$ on $\hh$ by $\mu, \nu$.
\end{itemize}
\end{thm}

\begin{proof}
If $(\hh_\lambda, \pp\dlam, f\dlam)$  is a crossed module over associative algebra, then $f\dlam$ is an equivariant map.
Thus we have
\begin{eqnarray*}
&&{f}\dlam (x\cdot_\lam a)-x\cdot_\lam {f}\dlam (a)\\
&=&({f}+\lambda\psi )( x\cdot {a} + \lambda\mu(x,{a}))-x({f} (a)+\lambda\psi (a))-\lambda\omega (x, {f} (a)+\lambda\psi (a))\\
&=&{f}(x\cdot {a})+\lambda\Big(\psi( x\cdot {a} )+{f}\mu(x,{a})\Big)+\lambda^2\psi \omega (x,{a})\\
&&-x f(a)-\lambda\Big(\omega(x, f(a)+x\psi (a)\Big)-\lambda^2\mu(x,\psi (a))\\
&=&0,
\end{eqnarray*}
which implies that
\begin{eqnarray}
\label{eq:2-coc01}\psi ( x\cdot {a} )-x\psi (a)+{f} \mu (x,{a})-\omega (x, {f} (a))&=&0,\\
\label{eq:2-coc02}\psi \mu(x,{a})-\omega (x,\psi (a))&=&0.
\end{eqnarray}
Similar computation shows that
\begin{eqnarray}
\label{eq:2-coc03}\psi ({a}\cdot x)-\psi (a)x+{f} \nu ({a},x)-\omega ({f} (a), x)&=&0,\\
\label{eq:2-coc04}\psi \nu({a},x)-\omega (\psi (a),x)&=&0.
\end{eqnarray}

Since $\pp\dlam$ is an associative algebra, we have
\begin{eqnarray*}
&&x\cdot_\lam (y\cdot_\lam z)-x\cdot_\lam  (y\cdot_\lam  z)\\
&=&x\cdot (y\cdot z)]-(x\cdot y)\cdot z\\
&&+\lambda\Big(\omega (x, y\cdot z)+[x,\omega (y,z)]-\omega (x\cdot y, z)-\omega (x, y)\cdot z\\
&&+\lambda^2\Big(\omega(x,\omega (y,z))-\omega(\omega (x,y), z)\Big)\\
&=&0,
\end{eqnarray*}
which implies that
\begin{eqnarray}
\label{eq:2-cocycle1'} \omega (x, y\cdot z)+x\cdot\omega (y,z)-\omega (x\cdot y, z)-\omega (x, y)\cdot z&=&0,\\
\label{eq:2-cocycle1''}\omega(x,\omega (y,z))-\omega(\omega (x,y), z)-\omega( y,\omega (x,z))&=&0.
\end{eqnarray}

Since $\hh_\lambda$ is a left module of $\pp\dlam$, we have
\begin{eqnarray*}
&&x\cdot_\lam(y\cdot_\lam {a})-(x\cdot_\lam y)\cdot_\lam {a}\\
&=&x\cdot (y\cdot {a})-(x\cdot y)\cdot {a} \\
\nonumber&&+\lambda\Big(x\mu(y,{a})+\mu(x, y\cdot {a})- \mu(x\cdot y,{a})-\omega(x,y)\cdot {a}\Big)\\
&&+\lambda^2\Big(\mu(x, \mu(y,{a}))-\mu(\omega (x,y),{a})\Big)\\
&=&0,
\end{eqnarray*}
which implies that
\begin{eqnarray}
\label{eq:2-coc31} x\cdot \mu(y,{a})+\mu(x, y\cdot {a})- \mu(x\cdot y,{a})-\omega(x,y)\cdot {a}&=&0,\\
\label{eq:2-coc32} \mu(x, \mu(y,{a}))-\mu(\omega (x,y),{a})-\mu( y, \mu(x,{a}))&=&0.
\end{eqnarray}
Similar computation shows that
\begin{eqnarray}
\label{eq:2-coc33}\nu({a}, (x\cdot y))+{a}\cdot\omega(x,y)-\nu({a}, x)\cdot y-\nu({a}\cdot x,  y)&=&0,\\
\label{eq:2-coc34}\nu({a},\omega (x,y))-\nu(\nu(x,{a}),  y)-\mu(x,\nu ({a},y))&=&0.
\end{eqnarray}

At last, since
\begin{eqnarray*}
&&{f}\dlam (a)\cdot\dlam n-{a}\cdot\dlam  {f}\dlam (b)\\
&=&({f} (a)+\lambda\psi (a))\cdot {b}+\lambda\mu ({f} (a)+\lambda\psi (a),{b})\\
&&-{a}\cdot ({f} (b)+\lambda\psi (b))-\lambda\nu ({a}, {f} (b)+\lambda\psi (b))\\
&=&{f} (a)\cdot {b}+\lambda(\psi (a)\cdot {b}+\mu({f} (a),{b}))+\lambda^2\nu(\psi (a),{b})\\
&&-{a}\cdot {f} (b)-\lambda(\nu({a}, {f} (b))+{a}\cdot \psi (b))-\lambda^2\nu({a},\psi (b))\\
&=&0,
\end{eqnarray*}
which implies that
\begin{eqnarray}
\label{eq:2-coc41}\psi (a)\trl n+\nu({f} (a),n)-{a}\trr \psi (b)-\nu({a},{f} (b))&=&0,\\
\label{eq:2-coc42}\mu(\psi (a),n)-\nu({a},\psi (b))&=&0.
\end{eqnarray}

By \eqref{eq:2-coc01},  \eqref{eq:2-coc31}, \eqref{eq:2-coc33} and \eqref{eq:2-coc41}, we find that $(\psi,\omega,\mu,\nu)$ is a 2-cocycle of $(\hh,\pp,f)$ with the coefficients in the adjoint representation.
Furthermore, by \eqref{eq:2-coc02}, \eqref{eq:2-coc32}, \eqref{eq:2-coc34} and \eqref{eq:2-coc42}, $(\hh,\pp, \psi)$ is a crossed module over associative algebra.
\end{proof}

Now we introduce the notion of Nijenhuis operators which gives trivial deformations.

Let $(\hh,\pp, f)$ be a crossed module over associative algebra  and $N=(N_0,N_1)$ be a pair of linear maps $N_0:\pp \to \pp $ and $N_1: \hh\to  \hh $. Define an exact 2-cochain
 $$(\psi,\omega,\mu,\nu)=D(N_0,N_1)$$
 by differential $D$ discussed above, i.e.
  \begin{eqnarray*}
   \psi(a)&=&{f} \circ N_1(a)-N_0\circ {f}(a),\\
\omega(x, y)=(x\cdot_N y)&=&N_0(x)y +  x\cdot N_0(y) - N_0(x\cdot y),\\
\mu(x,{a})=x\cdot_N {a}&=&N_0(x) {a} + x\cdot N_1(a) - N_1(x\cdot {a}),\\
\nu({a},x)={a}\cdot_N x&=&N_1(a) x + {a}\cdot N_0(x) - N_1({a}\cdot x).
 \end{eqnarray*}

 \begin{defi}\label{defi:Nijenhuis}
  A pair of linear maps $N=(N_0,N_1)$ is called a Nijenhuis operator if for all $x,y\in\pp $ and $a\in \hh $, the following conditions are satisfied:
   \begin{itemize}
     \item[(i)] $f  N_1(a)=N_0 f(a),$
     \item[(ii)] $N_0(x\cdot_N y)=N_0(x)\cdot N_0(y),$
     \item[(iii)] $N_1(x\cdot_N {a})=N_0(x)\cdot N_1(a),$
     \item[(iv)] $N_1({a}\cdot_N x)=N_1(a)\cdot N_0(x).$
   \end{itemize}
 \end{defi}

\begin{defi}
A deformation is said to be trivial if there exists a pair of  linear maps $N_0:\pp \to \pp, ~N_1:  \hh \to  \hh $,
such that $(T_0,T_1)$ is a homomorphism from $(\hh\dlam,\pp\dlam,{f}\dlam)$ to $(\hh,\pp,{f}) $, where $T_0 = \id + \lambda N_0$,
$T_1 = \id + \lambda N_1$.
\end{defi}

\begin{thm}\label{thm:Nijenhuis}
Let $N=(N_0,N_1)$ be a Nijenhuis operator. Then a deformation
can be obtained by putting
\begin{equation}\label{Nijenhuis}
\left\{\begin{array}{rll}
\psi (a) &=&({f} N_1-N_0{f}) (a),\\
\omega (x, y)&=&N_0(x)\cdot y  +  x\cdot N_0(y) - N_0(x\cdot y),\\
\mu(x,{a})&=&N_0(x)\cdot {a} + x\cdot N_1(a) - N_1(x\cdot {a})\\
\nu({a},x)&=&{a}\cdot N_0(x) +  N_1(a)\cdot x - N_1({a}\cdot x).
\end{array}\right.
\end{equation}
Furthermore, this deformation is trivial.
\end{thm}

\begin{proof} Since $(\psi,\omega,\mu,\nu)=D(N_0,N_1)$, it is obvious that $(\psi,\omega,\mu,\nu)$ is a 2-cocycle.
It is easy to check that $(\psi,\omega,\mu,\nu)$ defines a crossed module over associative algebra $(\hh,\pp,\psi)$  .
Thus by Theorem \ref{thm:deformation}, $(\psi,\omega,\mu,\nu)$ generates a deformation.
\end{proof}

\subsection{Abelian extensions}

In this section, we study abelian extensions of crossed modules over associative algebras.

\begin{defi}
Let $(\hh,\pp, f)$ be a crossed module over associative algebra. An extension of $(\hh,\pp, f)$ is a
short exact sequence such that $\mathrm{Im}(i_0)=\mathrm{Ker}(p_0)$ and $\mathrm{Im}(i_1)=\mathrm{Ker}(p_1)$ in the following commutative diagram
\begin{equation}\label{eq:ext2}
\CD
  0 @>0>>  \frkh @>i_1>> \widehat{\hh} @>p_1>>  \hh  @>0>> 0 \\
  @V 0 VV @V \varphi VV @V \widehat{f} VV @V{f} VV @V0VV  \\
  0 @>0>> W @>i_0>> \widehat{\pp} @>p_0>> \pp @>0>>0
\endCD
\end{equation}
where $(V,W,\varphi)$ is a crossed module over associative algebra.
\end{defi}

We call $(\widehat{\hh},\widehat{\pp}, \widehat f)$ an extension of $(\hh,\pp, f)$ by
$(V,W,\varphi)$, and denote it by $\widehat{\E}$.
It is called an abelian extension if $(V,W,\varphi)$ is an abelian crossed module over associative algebra  (this means that the multiplication on $W$ is zero, the representation of $W$ on $V$ is trivial).

A splitting $\sigma=(\sigma_0,\sigma_1):(\hh,\pp, f)\to (\widehat{\hh},\widehat \pp, \widehat f)$ consists of linear maps
$\sigma_0:{\pp}\to\widehat{\pp}$ and $\sigma_1:\hh\to \widehat{\hh}$
 such that  $p_0\circ\sigma_0=\id_{{\pp}}$, $p_1\circ\sigma_1=\id_{\hh}$ and $\widehat{f}\circ\sigma_1=\sigma_0\circ f$.

 Two extensions of Hom-associative algebras
 $\widehat{\E}:0\longrightarrow(V,W,\varphi)\stackrel{i}{\longrightarrow}(\widehat{\hh},\widehat \pp, \widehat f)\stackrel{p}{\longrightarrow}(\hh,\pp, f)\longrightarrow0$
 and  $\widetilde{\E}:0\longrightarrow(V,W,\varphi)\stackrel{j}{\longrightarrow}(\widetilde{\hh},\widetilde{\pp}, \widetilde{f})\stackrel{q}{\longrightarrow}(\hh,\pp, f)\longrightarrow0$ are equivalent,
 if there exists a homomorphism $F:(\widehat{\hh},\widehat \pp, \widehat f)\longrightarrow(\widetilde{\hh},\widetilde{\pp}, \widetilde{f})$  such that $F\circ i=j$, $q\circ
 F=p$.

Let $(\widehat{\hh},\widehat\pp, \widehat f)$ be an extension of $(\hh,\pp, f)$ by
$(V,W,\varphi)$ and $\sigma:(\hh,\pp, f)\to (\widehat{\hh},\widehat \pp, \widehat f)$ be a splitting.
Define the following maps:
\begin{equation}\label{extension:bimod}
\left\{\begin{array}{rlclcrcl}
\cdot:&{\pp}\otimes V&\longrightarrow& V,&& x\cdot v&\triangleq&\sigma_0(x)\cdot v,\\
\cdot:&{\pp}\otimes W&\longrightarrow& W,&& x\cdot w&\triangleq&\sigma_0(x)\cdot w,\\
\cdot:&V\otimes{\pp}&\longrightarrow& V,&& v\cdot x&\triangleq&v\cdot\sigma_0(x),\\
\cdot:&W\otimes {\pp} &\longrightarrow& W,&& w\cdot x&\triangleq&w\cdot\sigma_0(x),\\
\trr:&W\otimes \hh&\longrightarrow&V,&&w\trl {a}&\triangleq&w\cdot\sigma_1(a),\\
\trl:&\hh\otimes W&\longrightarrow&V,&&{a}\trr w&\triangleq&\sigma_1(a)\cdot w,
\end{array}\right.
\end{equation}
for all $x,y,z\in\pp$, $m\in \hh $.

\begin{pro}\label{pro:2-modules}
With the above notations,  $(V,W,\varphi)$ is a representation of $(\hh,\pp,f)$.
Furthermore, this representation structure does not depend on the choice of the splitting $\sigma$.
Moreover,  equivalent abelian extensions give the same  representation on $(V,W,\varphi)$.
\end{pro}

\begin{proof}  Firstly, we show that the representation is well-defined.
Since $\Ker p_{0} \cong W$, then for $w\in {W}$, we have $p_{0}(w)=0$.
By the fact that $(p_1,p_0)$ is a homomorphism between $(\widehat{\hh},\widehat\pp, \widehat f)$ and $(\hh,\pp,f)$, we get
$$p_{1}(x\cdot w)=p_{1}(\sigma_0(x) w)=p_{0}\sigma_0(x)\cdot p_{0}(w)=p_{0}\sigma_0(x)\cdot 0=0.$$
Thus $x\cdot w\in  \ker p_{0} \cong {W}$.
Similar computations show that
$$p_{0}(x\cdot v)=p_{0}(\sigma_0(x)\cdot v)=p_{0}\sigma_0(x)\cdot p_{1}(v)=p_{0}\sigma_0(x)\cdot 0=0,$$
$$p_{0}({a}\trr w)=p_{0}(\sigma_1(a)\cdot w)=p_{0}\sigma_1(a)\cdot p_{0}(w)=p_{0}\sigma_0(a)\cdot 0=0.$$
Thus $x\cdot v, {a}\trr w\in \Ker p_0=V$.

Now we  will show that these maps are independent of the choice of $\sigma$. In fact, if we choose another splitting
$\sigma':\pp\to\widehat{\pp}$, then $p_0(\sigma_0(x)-\sigma'_0(x))=x-x=0$,
$p_1(\sigma_1(a)-\sigma'_1(a))={a}-{a}=0$, i.e. $\sigma_0(x)-\sigma'_0(x)\in
\Ker p_0=W$, $\sigma_1(a)-\sigma'_1(a)\in \Ker p_1=V$. Thus,
$(\sigma_0(x)-\sigma'_0(x)(v+w)=0$, $(\sigma_1(a)-\sigma'_1(a))w=0$, which implies that the maps in \eqref{extension:bimod}  are independent
on the choice of $\sigma$. Therefore the representation structures are well-defined.


Now $\varphi$ is an equivariant map since
\begin{eqnarray}
\varphi(x\cdot v)=\varphi(\sigma_0(x)\cdot v)=\sigma_0(x)\cdot\varphi(v)=x\cdot\varphi(v).
\end{eqnarray}
For $\trr: W\otimes \hh\to V$, we have
\begin{eqnarray}
\varphi( w\trl {a})&=&\varphi(w\cdot\sigma_1(a))=w\cdot\widehat{f}\sigma_1(a)\notag\\
&=&w\cdot\sigma_0(f(a))=w\cdot f(a).
\end{eqnarray}
For $\trl: \hh\otimes W\to V$, we have
\begin{eqnarray}
\varphi( {a}\trr w)&=&\varphi(\sigma_1(a)\cdot w)=\widehat{f}\sigma_1(a)\cdot w\notag\\
&=&\sigma_0(f(a))\cdot w=f(a)\cdot w.
\end{eqnarray}
One check that these two maps  $\trr$ and $\trl$ satisfying  the conditions \eqref{deflm11}--\eqref{deflm16}.
Therefore, $(V,W,\varphi)$ is a representation of $(\hh,\pp,f)$.

Finally, suppose that $\widehat{\E}$ and $\widetilde{\E}$ are equivalent abelian extensions,
and $F:(\widehat{\hh},\widehat \pp, \widehat f)\longrightarrow(\widetilde{\hh},\widetilde{\pp}, \widetilde{f})$ be the homomorphism.
Choose linear sections $\sigma$ and $\sigma'$ of $p$ and $q$. Then we have $q_0F_0\sigma_0(x)=p_0\sigma_0(x)=x=q_0\sigma'_0(x)$,
thus $F_0\sigma_0(x)-\sigma'_0(x)\in \Ker q_0=W$. Therefore, we obtain
$$\sigma'_0(x)\cdot w=F_0\sigma_0(x)\cdot w=F_0(\sigma_0(x)\cdot w)=\sigma_0(x)\cdot w,$$
which implies that equivalent abelian extensions give the same module structures on $W$.
Similarly, we can show that equivalent abelian extensions also give the same $(V,W,\varphi)$.
Therefore, equivalent abelian extensions also give the same representation. The proof is finished.
\end{proof}

Let $\sigma:(\hh,\pp, f)\to (\widehat{\hh},\widehat \pp, \widehat f)$  be a splitting of an abelian extension. Define the following linear maps:
\begin{equation}\label{extension:cocycle}
\left\{
\begin{array}{rlclcrcl}
\psi:& \hh &\longrightarrow& W ,&& \psi(a)&\triangleq&\widehat{f}\sigma_1(a)-\sigma_0(f(a)),\\
\omega:&\pp \otimes\pp&\longrightarrow& W ,&& \omega(x,y)&\triangleq&\sigma_0(x)\cdot\sigma_0(y)-\sigma_0(x\cdot y),\\
\mu:& {\pp}\otimes \hh&\longrightarrow&V,&& \mu(x,{a})&\triangleq&\sigma_0(x)\cdot\sigma_1(a)-\sigma_1(x\cdot {a}),\\
\nu:&\hh\otimes \pp &\longrightarrow&V,&& \nu({a},x)&\triangleq&\sigma_1(a)\cdot\sigma_0(x)-\sigma_1({a}\cdot x).
\end{array}\right.
\end{equation}
for all $x,y,z\in\pp$, $m\in \hh $.

\begin{thm}\label{thm:2-cocylce}
With the above notations, $(\psi,\omega,\mu,\nu)$ is a $2$-cocycle of $(\hh,\pp, f)$ with coefficients in $(V,W,\varphi)$.
\end{thm}

\begin{proof}
Since $\widehat{f}$ is an equivariant map, we have the equality
$$\widehat{f}\big(\sigma_0(x)\cdot\sigma_1(a))\big)=\sigma_0(x)\widehat{f}\sigma_1(a),$$
then by \eqref{extension:cocycle} we get that
\begin{eqnarray*}
\widehat{f}\big(\mu(x,{a})+\sigma_1( x\cdot {a} )\big)=\sigma_0(x)\big(\sigma_0(f(a))+\psi(a)\big),
\end{eqnarray*}
\begin{eqnarray*}
\varphi(\mu(x,{a}))+\psi( x\cdot {a} )+\sigma_0(f( x\cdot {a} ))= \omega(x,f(a))+\sigma_0(xf(a))+x\cdot \psi(a).
\end{eqnarray*}
Thus we obtain
\begin{eqnarray}\label{eq:c1}
x\cdot \psi(a)+\omega(x,f(a))&=&\psi( x\cdot {a} )+\varphi(\mu(x,{a})).
\end{eqnarray}

 Since $\widehat{\pp}$ is an associative algebra,
$$
 (\sigma_0(x))\big(\sigma_0(y)\sigma_0(z)\big)=\big(\sigma_0(x)\sigma_0(y)\big) (\sigma_0(z)),
$$
then we get
\begin{eqnarray}
{x\cdot \omega (y, z)+\omega (x, y\cdot z)}=\omega(x\cdot y,  z)+\omega (x, y)\cdot  z.\label{eq:c3}
\end{eqnarray}
Since $\widehat{\hh}$ is a left module of $\widehat{\pp}$,  we have the equality
\begin{eqnarray*}
&& \sigma_0(x)\cdot(\sigma_0(y)\cdot\sigma_1(a))-(\sigma_0(x)\sigma_0(y))\cdot( \sigma_1(a))-\\
&=&(\sigma_0x)\cdot\big(\mu(y,{a})+\sigma_1(y\cdot {a})\big)-\big(\omega(x,y)+\sigma_0(x\cdot y)\big)\cdot(\sigma_1(a))\\
&=&\sigma_1(x\cdot (y\cdot {a}))+\mu(x,y\cdot {a})+x\cdot \mu(y,{a})\\
&&-\sigma_1((x\cdot y)\cdot {a})-\mu(x\cdot y,{a})-\omega(x,y)\trl {a}\\
&=&0.
\end{eqnarray*}
Thus, we get
\begin{eqnarray}
\notag{x\cdot\mu(y,{a})+\mu(x,y\cdot {a})}&=&\omega(x,y)\trl {a}+\mu(x\cdot y,{a}).\label{eq:c41}
\end{eqnarray}
Similarly, since $\widehat{\hh}$ is a right module of $\widehat{\pp}$,  we obtain
\begin{eqnarray}
{a} \trr\omega(x,y)+\nu({a}, (x\cdot y))&=&\nu({a},x)\cdot  y+\nu({a}\cdot x, y).\label{eq:c42}
\end{eqnarray}
By the equality
\begin{eqnarray*}
&& \sigma_0(x)\cdot(\sigma_1(a)\sigma_0(y))-(\sigma_0(x)\cdot\sigma_1(a))\cdot  \sigma_0(y)\\
&=&(\sigma_0x)\cdot\big(\mu(y,{a})+\sigma_1(y\cdot {a})\big)-(\sigma_1(x\cdot {a})+\mu(x,{a}))\cdot\sigma_0(y)\\
&=&\sigma_1(x\cdot ({a}\cdot y))+\mu(x, {a}\cdot y)+x\cdot \nu({a},y)\\
&&-\sigma_1((x\cdot {a})\cdot  y)-\mu(x\cdot {a}, y)-\mu(x,{a})\cdot  y\\
&=&0,
\end{eqnarray*}
we get
\begin{eqnarray}
\nonumber\mu(x, {a}\cdot y)+x\cdot \nu({a},y)&=&\nu(x\cdot {a},  y)+\mu(x,{a})\cdot y\label{eq:c43}
\end{eqnarray}

Finally, since $\widehat{f}$ is the structure map of the crossed module, we have the equality
$$\widehat{f}(\sigma_1(a))\cdot\sigma_1(b)=\sigma_1(a)\cdot\widehat{f}(\sigma_1(b)),$$
then by
$$\psi(a)\triangleq\widehat{f}\sigma_1(a)-\sigma_0(f(a))\in W,$$
 we get that
\begin{eqnarray*}
&&(\psi(a)+\sigma_0(f(a)))\cdot\sigma_1(b)\\
&=&\psi(a)\cdot\sigma_1(b)+\sigma_0(f(a))\cdot\sigma_1(b)\\
&=&\psi(a)\cdot\sigma_1(b)+\mu(f(a), n)-\sigma_1(f(a)\cdot n)
\end{eqnarray*}
and
\begin{eqnarray*}
&&\sigma_1(a)\cdot (\psi(b)+\sigma_0(f(b)))\\
&=&\sigma_1(a)\cdot \psi(b)+\sigma_1(a)\cdot \sigma_0(f(b))\\
&=&\sigma_1(a)\cdot \psi(b)+ \nu({a},f(b))+\sigma_1({a}\cdot f(b)).
\end{eqnarray*}
Thus we obtain
\begin{eqnarray}
\psi (a)\trl {b}+\nu({f} (a),{b})-{a}\trr \psi (b)-\nu({a},{f} (b))=0.
\end{eqnarray}

Therefore, by the above discussion, we obtain that $(\psi,\omega,\mu,\nu)$ is a $2$-cocycle of $(\hh,\pp, f)$ with coefficients in $(V,W,\varphi)$.
This complete the proof.
\end{proof}

Now we define a crossed module over associative algebra on $(\hh \oplus V, \pp \oplus W, \widehat{f})$ using the 2-cocycle given above.
More precisely, we have
\begin{equation}\label{eq:multiplication}
\left\{\begin{array}{rcl}
\widehat{f}({a}+v)&\triangleq&\dM({a})+\psi(v)+\varphi(v),\\
{(x+w)(x'+w')}&\triangleq&x\cdot x'+\omega(x,x')+x\cdot w'+w\cdot x',\\
{(x+w)\cdot ({a}+v)}&\triangleq& x\cdot {a} +\mu(x,{a})+x\cdot v+ w\trl {a},\\
{({a}+v)\cdot (x+w)}&\triangleq& {a}\cdot x +\nu({a},x)+v\cdot x+ {a}\trr w,
\end{array}\right.
\end{equation}
for all $x,y,z\in {\pp}$, ${a},{b}\in \hh$,
$v\in V$ and $w\in W$. Thus any extension
$\widehat{E}$ given by \eqref{eq:ext1} is isomorphic to
\begin{equation}\label{eq:ext2}
\CD
  0 @>0>>  V @>i_1>>  \hh \oplus V@>p_1>>  \hh  @>0>> 0 \\
  @V 0 VV @V \varphi VV @V \widehat{{f}} VV @V{f} VV @V0VV  \\
  0 @>0>>  W  @>i_0>> \pp \oplus W @>p_0>> \pp @>0>>0,
\endCD
\end{equation}
where the crossed module over associative algebra is given as above \eqref{eq:multiplication}.

\begin{thm}\label{mainthm512}
There is a one-to-one correspondence between equivalence classes of abelian extensions and the second cohomology group ${H}^2((\hh,\pp,f),(V,W,\varphi))$.
\end{thm}

The proof of Theorem \ref{mainthm512} is similar as the proof of Theorem \ref{mainthm44}, so we omit the details.

\section*{Acknowledgements.}
This is a primary edition. Something will be modified in the future.

\vskip7pt
\footnotesize{
\noindent Tao Zhang\\
College of Mathematics and Information Science,\\
Henan Normal University, Xinxiang 453007, P. R. China;\\
 E-mail address: \texttt{{zhangtao@htu.edu.cn}}

\end{document}